\documentclass[10pt,reqno]{amsart} 
\usepackage[a4paper, margin=2.3cm]{geometry}

\usepackage{bbm}
\usepackage{mathrsfs}

\usepackage{enumitem}

\usepackage{float}
\usepackage{amsbsy}
\usepackage[all]{xy}\usepackage{xypic}
\usepackage{braket}
\usepackage[colorlinks = true,citecolor=black]{hyperref}

\hypersetup{
     colorlinks=true,
     linkcolor=blue,
     filecolor=blue,
     citecolor = blue,      
     urlcolor=cyan,
     }

\DeclareMathAlphabet{\mathpzc}{OT1}{pzc}{m}{it}
\usepackage{tikz-cd}
\usepackage{tkz-graph}
\usetikzlibrary{shapes.geometric,arrows, shapes}

\newtheorem{theorem}{Theorem}[section]
\newtheorem*{theorem*}{Theorem}
\newtheorem{theorem-non}{Theorem}
\newtheorem{lemma-non}{Lemma}

\newtheorem{corollary-non}{Corollary}

\newtheorem*{lemma*}{Lemma}
\newtheorem{corollary}[theorem]{Corollary}

\newtheorem*{conjecture*}{Conjecture}

\theoremstyle{definition}
\newtheorem{definition}[theorem]{Definition}

\theoremstyle{remark}
\newtheorem{remark}[theorem]{Remark}
\newtheorem{claim}{Claim}
\newtheorem{question}[theorem]{Question}

\numberwithin{equation}{section}

\begin{document}
\title[From Complex contact structures to real almost contact 3-structures]{From Complex contact structures to real almost contact 3-structures}

\author{Eder M. Correa}

\thanks{Eder M. Correa was supported by PRPq-UFMG grant 27764*32}

\address{{UFMG, Avenida Ant\^{o}nio Carlos, 6627, 31270-901 Belo Horizonte - MG, Brazil}}
\address{E-mail: {\rm edermc@ufmg.br.}}

\begin{abstract} 
In this work, we prove that every complex contact structure gives rise to a distinguished type of almost contact metric 3-structure. As an application of our main result, we provide several new examples of manifolds which admit taut contact circles, taut and round almost cosymplectic 2-spheres, and almost hypercontact (metric) structures. These examples generalize, in a suitable sense, the well-known examples of contact circles defined by the Liouville-Cartan forms on the unit cotangent bundle of Riemann surfaces. Furthermore, we provide sufficient conditions for a compact complex contact manifold to be the twistor space of a positive quaternionic K\"{a}hler manifold. In the particular setting of Fano contact manifolds, from our main result, we also obtain new evidences supporting the LeBrun-Salamon conjecture.
\end{abstract}

\maketitle

\hypersetup{linkcolor=black}

\hypersetup{linkcolor=black}

\section{Introduction}

In 1959, Kobayashi \cite{Kobayashi} introduced the notion of complex contact manifolds. Among other results, he showed that every complex contact manifold $(Z, \mathscr{J},\theta)$ is base space of a ${\rm{U}}(1)$-principal bundle $Q$ endowed with a (distinguished) real contact structure. Unlike standard examples of real contact structures on principal circle bundles, such as the well-known Boothby-Wang fibrartions \cite{Boothby}, Kobayashi's real contact structure does not come from a ${\rm{U}}(1)$-principal connection. In fact, it is naturally obtained from the complex contact structure defined on the base space. On the other hand, since the ${\rm{U}}(1)$-principal bundle $Q$ in Kobayashi's construction is given by the unitary frame bundle of a holomorphic line bundle $L \in {\text{Pic}}(Z)$, its Euler class is of $(1,1)$-type, hence, from Hatakeyama's result \cite{HATAKEYAMA}, it follows that $Q$ also can be endowed with a real normal almost contact structure \cite{Sasaki}, \cite{SASAKIHARAKEYMAALMOST}. Based on these constructions, a natural question which arises is: 

\begin{question}
\label{Q1}
What is the relationship between Kobayashi's real contact structure and Hatakeyama's real almost contact structure? 
\end{question}
Motivated by the above question, and by the ideas introduced in \cite{IshiharaKonishi1}, we investigate the compatibility of such structures, and give an answer for Question \ref{Q1}. Our main result shows that Kobayashi's real contact structure and Hatakeyama's real almost contact structure are compatible in the sense of almost 3-contact geometry \cite{Kuo}. More precisely, we prove the following theorem:

\begin{theorem-non}
\label{T1}
Let $(Z, \mathscr{J},\theta)$ be a complex contact manifold of complex dimension $2n+1 \geq 3$. Then there exists a $\rm{U}(1)$-principal bundle $Q$ over $Z$ which admits an almost 3-contact metric structure $(g_{Q},\Phi_{\alpha},\xi_{\alpha},\eta_{\alpha})$, $\alpha = 1,2,3$, satisfying the following properties:
\begin{enumerate}
\item $(\Phi_{1},\xi_{1},\eta_{1})$ is a normal almost contact structure, such that $Z = Q/\mathscr{F}_{\xi_{1}}$, and $\mathscr{L}_{\xi_{1}}g_{Q} = 0$;
\item $\eta_{2}$ and $\eta_{3}$ are contact structures, such that $\eta_{2} \wedge (d\eta_{2})^{2n+1} = \eta_{3} \wedge (d\eta_{3})^{2n+1} \neq 0$;
\item $(g_{Q},\Phi_{\alpha},\xi_{\alpha},\eta_{\alpha})$ is a contact metric structure, for $\alpha = 2,3$.
\end{enumerate}
Moreover, both $Q$ and $(g_{Q},\Phi_{\alpha},\xi_{\alpha},\eta_{\alpha})$, $\alpha = 1,2,3$, can be constructed in a natural way from $Z$ and $\theta$. 
\end{theorem-non}

The result above shows that every complex contact structure gives rise to a distinguished type of almost 3-contact metric structure. It is worth pointing out that the almost contact metric 3-structure given in the above theorem can be obtained in a constructive way from $(Z, \mathscr{J},\theta)$. Let us briefly outline the mains steeps in the proof of Theorem \ref{T1}. Firstly, we observe that the structure tensors $(\Phi_{1},\xi_{1},\eta_{1})$ in Theorem \ref{T1} are given by Hatakeyama's almost contact structure, so we have $Z = Q/\mathscr{F}_{\xi_{1}}$, i.e., $\xi_{1}  \in \mathfrak{X}(Q)$ generates the ${\rm{U}}(1)$-action on $Q$, and the normality condition in item (1) is a consequence of \cite[Theorem 2]{HATAKEYAMA}. The contact form $\eta_{2}$ is given by Kobayashi's real contact structure, it is obtained from the real part of the holomorphic $1$-form $\pi^{\ast}\theta$, where\footnote{Here $L \in {\text{Pic}}(Z)$ satisfies $L^{\otimes(n+1)}= K_{Z}$, and $L^{\times} = L \backslash \{{\text{zero section}}\}$.} $\pi \colon L^{\times} \to Z$ denotes the bundle projection, see \cite{Kobayashi} for more details. From the complex almost\footnote{Every complex contact manifold admits a complex almost contact structure, see for instance \cite{IshiharaKonishi}.} contact structure induced by $\theta$ on $Z$, we show that one can construct suitable tensor fields $\Phi_{2} \in {\text{End}}(TQ)$, and $\xi_{2} \in \mathfrak{X}(Q)$, satisfying 
\begin{equation}
\Phi_{2} \circ \Phi_{2} = - {\rm{Id}} + \eta_{2} \otimes \xi_{2}, \ \ \ \ \eta_{2}(\xi_{2}) = 1,
\end{equation}
i.e., in such a way that $(\Phi_{2},\xi_{2},\eta_{2})$ defines another almost contact structure on $Q$. Then, we show that $(\Phi_{\alpha},\xi_{\alpha},\eta_{\alpha})$, $\alpha = 1, 2$, satisfy the following relations:
\begin{equation}
\Phi_{1}(\xi_{2}) = -\Phi_{2}(\xi_{1}), \ \ \eta_{1} \circ \Phi_{2} = - \eta_{2} \circ \Phi_{1}, \ \ \eta_{1}(\xi_{2}) = \eta_{2}(\xi_{1}) = 0,
\end{equation}
\begin{equation}
\Phi_{1} \circ \Phi_{2} - \eta_{2} \otimes \xi_{1} = - \Phi_{2} \circ \Phi_{1} + \eta_{1} \otimes \xi_{2}.
\end{equation}
From the relations above, by applying \cite[Theorem 1]{Kuo}, we obtain an almost contact 3-structure $(\Phi_{\alpha},\xi_{\alpha},\eta_{\alpha})$, $\alpha = 1, 2,3$, on $Q$. In order to verify item (2), we show that $\eta_{3}$ coincides with the imaginary part of the holomorphic 1-form $\pi^{\ast}\theta$, and, from a similar computation as in \cite{Kobayashi}, we show that $\eta_{3} \wedge (d\eta_{3})^{2n+1} = \eta_{2} \wedge (d\eta_{2})^{2n+1} \neq 0$. The Riemannian metric $g_{Q}$ in Theorem \ref{T1}, is defined by
\begin{equation}
g_{Q} := \pi^{\ast}_{Q}(g_{Z}) + \eta_{1} \otimes \eta_{1},
\end{equation}
where $g_{Z}$ is a suitable Hermitian metric on $Z$ which is associated\footnote{The existence of such an associated Hermitian metric on $(Z, \mathscr{J},\theta)$ was shown by Ishihara and Konishi in \cite{IshiharaKonishi}, see also \cite{Foreman}.} to the complex almost contact structure induced by $\theta$. From the above definition, we have that $g_{Q}$ is compatible, in the sense of almost contact metric geometry, with $(\Phi_{1},\xi_{1},\eta_{1})$, and $\mathscr{L}_{\xi_{1}}g_{Q} = 0$, cf. \cite[Theorem 1]{HATAKEYAMA}. Using the fact that the Hermitian metric $g_{Z}$ is associated to the underlying complex almost contact structure on the base space $Z$, we conclude the proof of the result stated in Theorem \ref{T1} by showing that $g_{Q}$ is compatible with $(\Phi_{\alpha},\xi_{\alpha},\eta_{\alpha})$, $\alpha = 2,3$, in the sense of contact metric geometry. 

\begin{remark}
In view of the ideas (briefly) described above, let us point out two important facts which show why the almost contact metric 3-structure provided by Theorem \ref{T1} is not unique:
\begin{itemize}
\item Although the contact structures $\eta_{\alpha}$, $\alpha = 2,3$, are uniquely determined by $\theta$, the $(1,1)$-tensor fields $\Phi_{\alpha}$, $\alpha = 2,3$, depend on the choice of a ${\rm{U}}(1)$-principal connection $ \sqrt{-1}\eta_{1}$ (horizontal lift) on $Q$, and this choice is not unique. Thus, the almost contact 3-structure $(\Phi_{\alpha},\xi_{\alpha},\eta_{\alpha})$, $\alpha = 1,2,3$, is not uniquely determined;

\item Besides the choice of a principal connection on $Q$, the Riemannian metric $g_{Q}$ also depends on the choice of an associated Hermitian metric $g_{Z}$ on the base space, and this choice also is not unique in general \cite{IshiharaKonishi}, \cite{Foreman}.
\end{itemize}
\end{remark}

\begin{remark}
It is worth pointing out that, according to the ideas introduced in \cite{IshiharaKonishi1}, the converse of Theorem \ref{T1} also seems to be true, i.e., an almost contact metric 3-structure which satisfies the properties (1)-(3) of Theorem \ref{T1} gives rise to a complex contact structure.
\end{remark}

As we shall see bellow, the result provide by Theorem \ref{T1} has some interesting consequences. The first one is the following result:

\begin{corollary-non}
\label{C1}
Under the hypotheses of Theorem \ref{T1}, for every $s = (a,b,c) \in S^{2}$, we have an almost contact metric structure $(g_{Q},\Phi_{s},\xi_{s},\eta_{s})$ on $Q$, such that 
\begin{equation}
\Phi_{s} = a\Phi_{1} + b\Phi_{2} + c\Phi_{3}, \ \ \ \xi_{s} = a\xi_{1} + b\xi_{2} + c\xi_{3}, \ \ \ \eta_{s} = a\eta_{1} + b\eta_{2} + c\eta_{3}.
\end{equation}
Moreover, by considering $\nu_{s} : = g_{Q}(\Phi_{s}\otimes{\rm{Id}})$, we have $\eta_{s} \wedge (\nu_{s})^{2n+1} = \eta_{s'} \wedge (\nu_{s'})^{2n+1} \neq 0$, for all $s,s' \in S^{2}$.
\end{corollary-non}
Contact circles and contact p-spheres are families of contact forms parameterized, respectively, by the circle and the p-sphere, theses concepts were introduced by H. Geiges and J. Gonzalo in \cite{GeigesGonzalo}, see also \cite{Zessin}. As it can be observed from Theorem \ref{T1}, and Corollary \ref{C1}, Kobayashi's contact structure can be realized as an element of the (taut\footnote{That is, the volume forms $\eta_{s} \wedge (d\eta_{s})^{2n+1}$ are equal for every $s \in S^{2} \cap \{a = 0\}$.}) contact circle
\begin{equation}
\big \{ \eta_{s}  \in \Omega^{1}(Q)\ \big| \ s \in S^{2} \cap \{a = 0\} \ \big \}.
\end{equation}
In the literature, the structure described in Corollary \ref{C1} is also known as {\textit{almost hypercontact (metric) structure}} \cite{Sasakigeometry}. Further, following \cite[Corollary 4.4]{Cappeletti}, we notice that the family $\{(\eta_{s},\nu_{s})\}_{s \in S^{2}}$ defines an {\textit{almost cosymplectic 2-sphere}} on $Q$ which is round\footnote{It means that $\eta_{\alpha}(\xi_{\beta}) + \eta_{\beta}(\xi_{\alpha})  = 0$, $\forall \alpha,\beta \in \{1,2,3\}$, $\alpha \neq \beta$, and $\iota_{\xi_{\alpha}}\nu_{\beta} + \iota_{\xi_{\beta}}\nu_{\alpha} = 0$, $\forall \alpha,\beta \in \{1,2,3\}$.} and taut. 
\begin{remark}[Unit cotangent bundles]
\label{unitcotangent}
In the above setting, one obtains a huge class of new examples of taut contact circles and almost cosympletic 2-spheres by means of the following well-know construction: Let $M$ be any complex manifold, such that $\dim_{\mathbbm{C}}(M) = n+1$. Consider the tautological holomorphic 1-form $\Lambda$ defined on its cotangent bundle $T^{\ast}M$, i.e.,
\begin{equation}
\label{hsymplectic}
\Lambda(X) := \gamma(p_{\ast}(X)), \ \ \ \ X \in T_{\gamma}(T^{\ast}M),
\end{equation}
where $p \colon T^{\ast}M \to M$ is the natural projection. From this, we have a complex contact structure $\theta$ on the projective cotangent bundle $Z = {\text{P}}(T^{\ast}M)$, such that $\pi^{\ast}\theta =\Lambda|_{(T^{\ast}M)^{\times} }$, where $\pi \colon (T^{\ast}M)^{\times} \to {\text{P}}(T^{\ast}M)$ is the projection map\footnote{Notice that, in this particular case, we have $L^{\times} = \mathscr{O}_{{\text{P}}(T^{\ast}M)}(-1)^{\times} \cong (T^{\ast}M)^{\times}$, see for instance \cite{Kobayashibundles}, \cite{Kobayashi}.}. In this particular case, fixed any Hermitian metric on $T^{\ast}M$, we obtain from Theorem \ref{T1} that there exists an almost contact metric $3$-structure $(g_{Q},\Phi_{\alpha},\xi_{\alpha},\eta_{\alpha})$, $\alpha = 1,2,3$, on the unit cotangent bundle $Q = {\text{S}}^{1}(T^{\ast}M)$, which is completely determined by $({\text{P}}(T^{\ast}M), \theta)$. By applying Corollary \ref{C1}, we get several examples of taut contact circles, taut and round almost cosymplectic 2-spheres, and almost hypercontact (metric) structures. It is worth observing that this last construction generalizes, in a suitable sense, the well-know example of contact circles defined by the Liouville-Cartan forms on the unit cotangent bundle ${\text{S}}^{1}(T^{\ast}\Sigma)$ of Riemann surface $\Sigma$, cf. \cite{GeigesGonzalo}, \cite[Section 2.1]{Albers}. 
\begin{remark}
We also observe that, if $Z$ is a projective contact manifold, such that $b_{2}(Z) \geq 2$, then $Z = {\text{P}}(T^{\ast}M)$, for some projective manifold $M$, see for instance \cite[Corollary 4]{Demailly}. Thus, in view of Remark \ref{unitcotangent}, from Theorem \ref{T1} we have that the relationship between projective contact manifolds and almost 3-contact metric manifolds goes beyond the well-known interplay between twistor spaces of positive quaternionic K\"{a}hler manifolds and 3-Sasakian manifolds, see for instance \cite{Ishihara}, \cite{Konishi}, \cite{IshiharaKonishi1}, \cite{3twistor}, \cite{Tanno}, \cite{Jelonek}. 
\end{remark}
\end{remark}
By using the almost contact metric 3-structure obtained from Theorem \ref{T1}, one also obtains sufficient conditions for a complex contact manifold to be K\"{a}hler-Einstein. Actually, we have the following result:
\begin{corollary-non}
\label{Corollary2}
In the setting of Theorem \ref{T1}, $(Z, \mathscr{J},\theta)$ admits a K\"{a}hler-Einstein metric with positive scalar curvature if at least one of the following (equivalent) two conditions holds:
\begin{enumerate}
\item $\Phi_{1} = \nabla \xi_{1}$, where $\nabla$ is the Levi-Civita connection of $g_{Q}$;
\item $\big [ \Phi_{\alpha},\Phi_{\alpha} \big ] + 2d\eta_{\alpha} \otimes \xi_{\alpha} = 0$, for $\alpha = 2$ or $\alpha = 3$.
\end{enumerate}
In particular, if $Z$ is compact, and $(1)$ or $(2)$ holds, then $(Z, \mathscr{J},\theta)$ is the twistor spaces of a compact positive quaternionic K\"{a}hler manifold.
\end{corollary-non}

The key point to obtain the result above is observing that, if at least one of the conditions in the above corollary holds, then the almost contact metric 3-structure $(g_{Q},\Phi_{\alpha},\xi_{\alpha},\eta_{\alpha})$, $\alpha = 1,2,3$, obtained from Theorem \ref{T1}, is in fact 3-Sasakian. In the particular setting of Fano contact manifolds, the problem related to the existence of K\"{a}hler-Einstein metrics is an open question which has important implications in the classification of Fano contact manifolds \cite{LEBRUNFANO}, \cite{BEAUVILLE}, \cite{Fortourists}. For this particular class of complex contact manifolds, we have that Theorem \ref{T1} takes the following form:
\begin{corollary-non}
\label{3contactfano}
Let $(Z, \mathscr{J},\theta)$ be a Fano contact manifold of complex dimension $2n+1 \geq 3$. Then there exists a $\rm{U}(1)$-principal bundle $Q$ over $Z$ which admits an almost contact metric 3-structure $(g_{Q},\Phi_{\alpha},\xi_{\alpha},\eta_{\alpha})$, $\alpha = 1,2,3$, satisfying the following properties:
\begin{enumerate}
\item $(\Phi_{1},\xi_{1},\eta_{1})$ is a normal almost contact structure, such that $Z = Q/\mathscr{F}_{\xi_{1}}$, and $\mathscr{L}_{\xi_{1}}g_{Q} = 0$;
\item $(\eta_{1},\eta_{2},\eta_{3})$ is a triple of contact structures, such that $\eta_{2} \wedge (d\eta_{2})^{2n+1} = \eta_{3} \wedge (d\eta_{3})^{2n+1} \neq 0$;
\item $(g_{Q},\Phi_{\alpha},\xi_{\alpha},\eta_{\alpha})$ is a contact metric structure, for $\alpha = 2,3$.
\end{enumerate}
Moreover, both $Q$ and $(\Phi_{\alpha},\xi_{\alpha},\eta_{\alpha})$, $\alpha = 1,2,3$, can be constructed in a natural way from $Z$ and $\theta$. 
\end{corollary-non}

In the above setting, we have that $(Q,\eta_{1})$ is the Boothby-Wang fibration defined by the Euler class ${\rm{e}}(Q) \in H^{2}(Z,\mathbbm{Z})$. We also observe that, under the assumption that $(Z, \mathscr{J},\theta)$ is Fano contact with $b_{2}(Z) > 1$, it follows that $Z = {\text{P}}(T^{\ast}\mathbbm{C}P^{n+1})$, see for instance \cite{LeBrunSalamon}, i.e., $Z$ is the twistor space of a homogeneous quaternionic K\"{a}hler manifold with positive scalar curvature (Wolf space \cite{Wolf}). Hence, $Z$ is also homogeneous \cite{Boothby2}. For the case that $b_{2}(Z) = 1$, we have the following conjecture:

\begin{conjecture*}[LeBrun-Salamon, \cite{LeBrunSalamon}, \cite{LEBRUNFANO} ]
\label{LSconhecture}
Let $Z$ be a Fano contact manifold with $b_{1}(Z) = 1$, then $Z$ must be homogeneous.
\end{conjecture*}

The conjecture above has been verified for low dimensional cases $n \leq 4$, see for instance \cite{Algtorus} and references therein. However, it is still an open problem in its full generality. It is worth to observe that, under the assumption that $(Z, \mathscr{J},\theta)$ is homogeneous, the manifold $Q$ in Corollary \ref{3contactfano} can be endowed with a 3-Sasakian structure. In fact, in this case $Q$ is the Konishi bundle \cite{Konishi} associated to a Wolf space \cite{Boothby2}, \cite{Wolf}. Therefore, in view of the above conjecture, it is expected that the $\rm{U}(1)$-principal bundle $Q$ in Corollary \ref{3contactfano} is in fact a 3-Sasakian manifold, for every Fano contact manifold $(Z, \mathscr{J},\theta)$. In this way, the result provided by our last result can be realized as an evidence\footnote{The most stronger evidence thus far of the validity of Conjecture \ref{LSconhecture} is the result provided in \cite{BEAUVILLE}.} in favor of the LeBrun-Salamon conjecture.  

Based on the interplay between 3-Sasakian manifolds and hyperk\"{a}hler manifolds \cite{Sasakigeometry}, the result of our main theorem can be translated to the language of almost hyperhermitian geometry in the following way:
\begin{corollary-non}
\label{C2}
Let $(Z, \mathscr{J},\theta)$ be a complex contact manifold of complex dimension $2n+1 \geq 3$. Then there exists a $\mathbbm{C}^{\times}$-principal bundle $\mathscr{U}(Z)$ over $Z$ such that $\mathscr{U}(Z)$ admits an almost hyperhermitian structure $(g_{\mathscr{U}},\mathbbm{I}_{1},\mathbbm{I}_{2},\mathbbm{I}_{3})$, satisfying:

\begin{enumerate}
\item $(g_{\mathscr{U}},\mathbbm{I}_{1})$ is a Hermitian structure, i.e., $[ \mathbbm{I}_{1},\mathbbm{I}_{1}] = 0$;

\item $\omega_{\alpha} = g_{\mathscr{U}}(\mathbbm{I}_{\alpha} \otimes {\rm{Id}})$, $\alpha = 2,3$, are symplectic structures;

\item $\Upsilon := \omega_{2} + \sqrt{-1}\omega_{3}$ is a holomprphic symplectic structure on $(\mathscr{U}(Z),\mathbbm{I}_{1})$.
\end{enumerate}
Furthermore, both $\mathscr{U}(Z)$ and $(g_{\mathscr{U}},\mathbbm{I}_{1},\mathbbm{I}_{2},\mathbbm{I}_{3})$ can be constructed in a natural way from $(Z, \mathscr{J},\theta)$. 
\end{corollary-non}

In the result above, we have $\mathscr{U}(Z) = {\text{Tot}}(L^{\times}) \cong Q \times \mathbbm{R}$, and the almost hyperhermitian structure $(g_{\mathscr{U}},\mathbbm{I}_{1},\mathbbm{I}_{2},\mathbbm{I}_{3})$ is obtained naturally from the almost contact metric $3$-structure provided by Theorem \ref{T1}. Further, by construction, we have $\Upsilon = d(\pi^{\ast}\theta)$, where $\pi \colon \mathscr{U}(Z) \to Z$ denotes the bundle projection\footnote{See for instance \cite[Lemma 1.2]{BEAUVILLE}.}. We notice that the previous corollaries also have an interpretation in terms of the almost hyperhermitian structure given in Corollary \ref{C2}. In fact, from the result above, for every $s = (a,b,c) \in S^{2}$, we have an almost Hermitian structure  $(g_{\mathscr{U}}, \mathbbm{I}_{s})$ on $\mathscr{U}(Z)$, such that 
\begin{equation}
 \mathbbm{I}_{s}(X) : = \Phi_{s}(X) - \eta_{s}(X)\frac{d}{dt}, \ \ \ \ \mathbbm{I}_{s}\Big ( \frac{d}{dt} \Big) := \xi_{s}, 
\end{equation}
such that $X \in \mathfrak{X}(Q)$, where $(\Phi_{s},\xi_{s},\eta_{s})$ is defined as in Corollary \ref{C1}. The sufficient conditions (1), and (2), given in Corollary \ref{Corollary2}, for $Z$ to be K\"{a}hler-Einstein, can be rephrased in terms of $(g_{\mathscr{U}},\mathbbm{I}_{1},\mathbbm{I}_{2},\mathbbm{I}_{3})$ as follows:
\begin{enumerate}
\item $d\omega_{1} = 0$, where $\omega_{1} = g_{\mathscr{U}}(\mathbbm{I}_{1} \otimes {\rm{Id}})$;
\item $[ \mathbbm{I}_{\alpha},\mathbbm{I}_{\alpha}] = 0$, for $\alpha = 2$ or $\alpha = 3$.
\end{enumerate}
If at least one of the above (equivalent) conditions is satisfied, then we have that $(g_{\mathscr{U}},\mathbbm{I}_{1},\mathbbm{I}_{2},\mathbbm{I}_{3})$ defines a hyperk\"{a}hler structure on $\mathscr{U}(Z)$, which in turn implies that $(Z, \mathscr{J},\theta)$ admits a K\"{a}hler-Einstein metric \cite{3Einstein}. In the particular case that $(Z, \mathscr{J},\theta)$ is homogeneous, we have that the manifold $\mathscr{U}(Z)$ in Corollary \ref{C2} can be endowed with a hyperk\"{a}hler structure. Actually, in this last case $\mathscr{U}(Z)$ is the Swann bundle \cite{Swann} over a Wolf space. As we can see, in view of Conjecture\ref{LSconhecture}, it is expected that the $\mathbbm{C}^{\times}$-principal bundle $\mathscr{U}(Z)$ is in fact a hyerk\"{a}hler manifold. Hence, under the hypotheses of Corollary \ref{3contactfano}, our last result also can be realized as an evidence supporting the LeBrun-Salamon conjecture. 
\begin{remark}[Cotangent bundles]
 Given a complex manifold $M$, in some particular cases, it can be shown that $T^{\ast}M$ admits a hyperk\"{a}hler metric which is compatible with the canonical holomorphic symplectic form $d\Lambda$ defined on $T^{\ast}M$ (Eq. \ref{hsymplectic}), see for instance \cite{CALABIANSATZ}, \cite{Nakajima}, \cite{Biquard}, \cite{Biquard1}, \cite{Kronheimer2}, and references therein. Also, in the case that $M$ is a real-analytic K\"{a}hler manifold, it was shown independently by D. Kaledin \cite{Kaledin}, and by B. Feix \cite{Feix}, that there exists a hyperk\"{a}hler metric in a neighbourhood of the zero section of $T^{\ast}M$ which is compatible with $d\Lambda$. In a broad sense, if  $Z = {\text{P}}(T^{\ast}M)$, for some complex manifold $M$, by applying Corollary \ref{C2}, we obtain an almost hyperhermitian structure $(g_{\mathscr{U}},\mathbbm{I}_{1},\mathbbm{I}_{2},\mathbbm{I}_{3})$ on the manifold $\mathscr{U}(Z) = (T^{\ast}M)^{\times}$, see Remark \ref{unitcotangent}. Moreover, in this case, the canonical complex structure on $(T^{\ast}M)^{\times}$ coincides with $\mathbbm{I}_{1}$, and the holomorphic symplectic form $\Upsilon$, given in item (3) of Corollary \ref{C2}, turns out to be the restriction to $(T^{\ast}M)^{\times}$ of the canonical holomorphic symplectic form $d\Lambda$ of $T^{\ast}M$. Therefore, for any complex manifold $M$, our last result shows that, at least outside of the zero section of $T^{\ast}M$, one can always find an almost hyperhermitian structure $(g_{\mathscr{U}},\mathbbm{I}_{1},\mathbbm{I}_{2},\mathbbm{I}_{3})$ compatible with the restriction of the canonical holomorphic symplectic structure of $T^{\ast}M$.
\end{remark}

\subsection*{Organization of the paper} This paper is organized as follows: In Section \ref{Sec2}, we provide a brief review about some generalities on complex contact manifolds, focusing on Kobayashi's construction of real contact structures and on its relationship with complex almost contact (metric) structures. In Section \ref{Sec3}, we review some basic generalities on almost contact manifolds, contact metric structures, and almost contact 3-structures. In Section \ref{Sec4}, we provide a complete proof for Theorem \ref{T1} and its Corollaries 1-4.

\section{Generalities on complex Contact manifolds} 
\label{Sec2}

In this section, we provide an overview on complex contact geometry. Our main purpose is to investigate Kobayashi's construction of real contact structures \cite{Kobayashi} from the view point of complex almost contact geometry \cite{IshiharaKonishi2}, \cite{IshiharaKonishi}.

\subsection{Complex contact manifolds} A {\textit{complex contact manifold}} is a complex manifold $(Z,\mathscr{J})$ with odd complex dimension $2n+1 \geq 3$ together with an open covering $\mathscr{U} = \{U_{i}\}_{i \in I}$ of coordinate neighborhoods such that:
\begin{enumerate}
\item On each open set $U_{i}$ we have a holomorphic 1-form $\theta_{i}$ such that
\begin{equation}
\label{nvanishing}
\theta_{i} \wedge \big (d\theta_{i} \big)^{n} \neq 0. 
\end{equation}
\item On $U_{i} \cap U_{j} \neq \emptyset$ there is a nonvanishing holomorphic function $f_{ij} \colon U_{i} \cap U_{j} \to \mathbbm{C}^{\times}$ such that 
\begin{center}
$\theta_{i} = f_{ij} \theta_{j}.$
\end{center}
\end{enumerate}
Given a complex contact manifold $(Z, \mathscr{J},\{\theta_{i}\})$, if $U_{i} \cap U_{j} \neq \emptyset$, since $\ker(\theta_{i})|_{U_{i} \cap U_{j}} = \ker(\theta_{j})|_{U_{i} \cap U_{j}}$, by defining
\begin{equation}
\mathscr{H}^{1,0} := \Big (\bigcup_{i \in I}\ker(\theta_{i}) \Big ) \bigcap T^{1,0}Z,
\end{equation}
it follows that $\mathscr{H}^{1,0}$ is a well-defined holomorphic subbundle of $T^{1,0}Z$ of maximal rank and complex dimension $2n$. This holomorphic subbundle is called {\textit{holomorphic contact subbundle}}. Considering the identification of holomorphic vector bundles $T^{1,0}Z  \cong (TZ,\mathscr{J})$, we shall denote by $\mathscr{H} \subset TZ$ the holomorphic subbundle corresponding to the contact subbundle $\mathscr{H}^{1,0}$.

From the line bundle $E = \{f_{ij}\} \in H^{1}(Z,\mathscr{O}_{Z}^{\ast})$, and the relation $\theta_{i} = f_{ij} \theta_{j}$ on the overlaps $U_{i} \cap U_{j} \neq \emptyset$, we can define a holomorphic vector bundle epimorphism $\theta \colon TZ \to E$, such that 
\begin{center}
$\theta|_{U_{i}} := \theta_{i} \otimes s_{i}, \ \ \ (\forall i \in I)$
\end{center}
where $s_{i} \in H^{0}(U_{i},E)$ is some nonvanishing holomorphic local section. From this, we obtain an exact sequence of holomorphic vector bundles
\begin{equation}
\label{exact}
0 \longrightarrow \ker(\theta) \longrightarrow  TZ \longrightarrow  E \longrightarrow  0,
\end{equation}
such that $\mathscr{H} = \ker(\theta)$. Moreover, from the condition $\theta_{i} = f_{ij} \theta_{j}$, we obtain
\begin{center}
$\theta_{i} \wedge \big (d\theta_{i} \big)^{n} = f_{ij}^{n+1}\theta_{j} \wedge \big (d\theta_{j} \big)^{n}$, 
\end{center}
on $U_{i} \cap U_{j} \neq \emptyset$, so we have a morphism $\theta \wedge (d\theta)^{n} \colon \det(TZ) \to E^{\otimes(n+1)}$. Since $\theta_{i} \wedge \big (d\theta_{i} \big)^{n}$ is a holomorphic $(2n+1)$-form which does not vanish in $U_{i}$, it follows that $\theta \wedge (d\theta)^{n}$ defines a isomorphism between $\det(TZ)$ and $E^{\otimes(n+1)}$. Hence, from the exact sequence \ref{exact}, we have
\begin{equation}
\label{cotactlinebundle}
E^{\otimes (n+1)} \cong \det(TZ) = \det(\mathscr{H}) \otimes E,
\end{equation}
so we obtain $\det(\mathscr{H}) \cong E^{\otimes n}$, and $K_{Z} = \{f_{ij}^{-(n+1)}\} \in H^{1}(Z,\mathscr{O}_{Z}^{\ast})$. The holomorphic line bundle $E$ is called {\textit{contact line bundle}}. 

\begin{remark}
In the construction above we have that $\theta  \in H^{0}(Z,  \Omega_{Z}^{1}\otimes E)$ defines completely the contact structure given by the local data $\{\theta_{i}\}$. Thus, we can also refer to a complex contact manifold as being a $(2n+1)$-dimensional complex manifold $(Z,\mathscr{J})$ with a twisted $1$-form $\theta  \in H^{0}(Z,  \Omega_{Z}^{1}\otimes E)$, such that $\theta \wedge (d\theta)^{n} \in H^{0}(Z, K_{Z} \otimes E^{\otimes(n+1)})$ does not vanish anywhere. 
\end{remark}

\begin{definition}
A Fano manifold is compact complex manifold $Z$ such that $c_{1}(Z)$ can be represented by a positive $(1,1)$-form.
\end{definition}

A Fano manifold with a complex contact structure is called {\textit{Fano contact manifold}}. This class of complex contact manifolds plays an important role in the study of positive quaternionic K\"{a}hler geometry. In fact, in \cite{SALAMONQUATERNIONIC}, is was shown that the twistor space of a compact positive quaternionic K\"{a}hler manifold is a complex contact manifold which admits a K\"{a}hler-Einstein metric with positive scalar curvature. Moreover, the converse is also true:

\begin{theorem}[ LeBrun, \cite{LEBRUNFANO}]
\label{LeBrun}
Let $Z$ be a Fano contact manidold. Then $Z$ is a twistor space iff it admits a K\"{a}hler-Einstein metric.
\end{theorem}
\begin{remark}
So far, it is unknown whether there are Fano contact manifolds which do not admit K\"{a}hler-Einstein metrics. Being more precise, if $Z$ is a Fano contact manifold, such that $b_{2}(Z) > 1$, then $Z = {\text{P}}(T^{\ast}\mathbbm{C}P^{n+1})$, see for instance \cite{LeBrunSalamon}, so it is a homogeneous K\"{a}hler-Einstein manifold. In the case that $Z$ is Fano contact with $b_{1}(Z) = 1$, it is conjectured (Conjecture \ref{LSconhecture}) that $Z$ is also homogeneous K\"{a}hler-Einstein manifold, i.e., it is conjectured that $Z = {\text{P}}(\mathcal{O}_{{\text{min}}})$, where $\mathcal{O}_{{\text{min}}} \subset \mathfrak{g}^{\mathbbm{C}}$ is the unique minimal nilpotent orbit associated to some complex simple Lie algebra $\mathfrak{g}^{\mathbbm{C}}$, see for instance \cite{BEAUVILLE},\cite{Boothby2}. 
\end{remark}
\subsection{Kobayashi's Real contact structure} 
\label{Kobayashi'scontact}
Let $(Z, \mathscr{J},\theta)$ be a complex contact manifold, and consider the holomorphic line bundle defined by $L = E^{-1}$, where $E \in {\text{Pic}}(Z)$ is the associated contact line bundle. Fixed a Hermitian structure $\langle \cdot \ , \cdot \rangle_{L} \colon L \times L \to \mathbbm{C}$, consider
\begin{center}
$Q(L) = \big \{ u \in L \ \big | \ \langle u , u \rangle^{\frac{1}{2}} = 1 \ \big \}.$ 
\end{center}
In \cite{Kobayashi}, Kobayashi showed that the complex contact structure $\theta  \in H^{0}(Z,  \Omega_{Z}^{1}\otimes E)$ induces a real contact form $\eta \in \Omega^{1}(Q(L))$ on the total space of the $\rm{U}(1)$-principal bundle $Q(L) \to Z$. This real contact structure can be described as follows: Firstly, notice that  $L = \{ g_{ij}\}$, such that $g_{ij} = f_{ij}^{-1}$. By taking holomorphic coordinates $\varphi_{i} \colon U_{i} \times \mathbbm{C}^{\times} \to L^{\times}$, $i \in I$, we can define
\begin{equation}
\label{localholomorphcform}
\vartheta_{i} :=z_{i} \pi^{\ast}\theta_{i}, \ \ \  \ \ (i \in I)
\end{equation}
where $\pi \colon L^{\times} \to Z$ is the projection map, and $z_{i} = pr_{2} \circ \varphi_{i}^{-1}$. Since $z_{i} = g_{ij}z_{j}$ on $L^{\times}|_{U_{i}\cap U_{j}}$, and $\theta_{i} = f_{ij}\theta$ on $U_{i}\cap U_{j}$, we have a globally well-defined holomorphic 1-form $\vartheta$ on ${\text{Tot}}(L^{\times})$, such that $\vartheta = \vartheta_{i}$ on $L|_{U_{i}}$, $\forall i \in I$. From $\vartheta \in \Omega^{1}_{{\text{Tot}}(L^{\times})}$, we define 
\begin{equation}
\label{Kobayashicontact}
\eta := \frac{1}{2}(\vartheta + \overline{\vartheta})|_{Q(L)} = \mathfrak{Re}(\vartheta)|_{Q(L)}.
\end{equation}
Using the local description of $\vartheta$, and the fact that $\theta_{i} \wedge \big (d\theta_{i} \big)^{n} \neq 0$, it can be shown that $\eta \wedge (d\eta)^{2n+1} \neq 0$. Thus, we have that $(Q(L),\eta)$ is a real contact manifold.

For our purpose, it will be useful to consider the following local description of $\eta \in \Omega^{1}(Q(L))$: Taking local coordinates $\varphi_{i} \colon U_{i} \times \mathbbm{C}^{\times} \to L^{\times}$, we have
\begin{equation}
\label{hermitian}
\langle \varphi_{i}(x,z_{i}),\varphi_{i}(x,z_{i})\rangle_{L} = h_{i}(x)|z_{i}|^{2}
\end{equation}
where $h_{i} \colon U_{i} \to \mathbbm{R}^{+}$ are positive smooth functions, satisfying 
\begin{equation}
\label{localhermitian}
h_{j} = h_{i}|g_{ij}|^{2}  \ \ \ \text{on}  \ \ \ U_{i}\cap U_{j} \neq \emptyset. 
\end{equation}
Thus, for every $\varphi_{i}(x,z_{i}) \in Q(L)|_{U_{i}}$, we have
\begin{equation}
\label{hermitianlocalrelation}
\displaystyle |z_{i}|  = \frac{1}{\sqrt{h_{i}(x)}}.
\end{equation}
Considering polar coordinates $z_{i} = |z_{i}|{\rm{e}}^{\sqrt{-1}\phi_{i}(z_{i})}$, we can describe $\eta$ (locally) as follows
\begin{equation}
\label{localcontactKobayashi}
\eta = \frac{\mathfrak{Re}\big({\rm{e}}^{\sqrt{-1}\phi_{i}} \pi_{Q}^{\ast}(\theta_{i})\big)}{\sqrt{h_{i} \circ \pi_{Q}}} = \cos(\phi_{i}) \pi_{Q}^{\ast} \Big (\frac{\mathfrak{Re}(\theta_{i})}{\sqrt{h_{i}}} \Big) - \sin(\phi_{i})\pi_{Q}^{\ast}\Big (\frac{\mathfrak{Im}(\theta_{i})}{\sqrt{h_{i}}} \Big),
\end{equation}
such that $\pi_{Q}  = \pi \circ \iota$, where $\iota \colon Q(L) \to {\text{Tot}}(L^{\times})$ is the natural inclusion map.

\begin{remark}
\label{transitionangle}
It will be important for us to consider the following generalities. From the definition, we have that $Q(L)$ can be described in terms of its transition functions $t_{ij} \colon U_{i} \cap U_{j} \to \rm{U}(1)$, such that $t_{ij} = \frac{g_{ij}}{|g_{ij}|}$. From this, we have
\begin{equation}
\label{polarcocylce}
t_{ij} = \cos(\psi_{ij}) - \sqrt{-1}\sin(\psi_{ij})={\rm{e}}^{-\sqrt{-1}\psi_{ij}}
\end{equation}
It is worth observing that $\phi_{j} = \phi_{i} + \psi_{ij} \circ \pi_{Q} + 2\pi k$, 
on $ Q(L)|_{U_{i} \cap U_{j}}$, with $k \in \mathbbm{Z}$.
\end{remark}

\subsection{Complex almost contact structures} 
\label{Cplxalmostsection}
Given a complex manifold $(Z, \mathscr{J})$, together with an open covering $\mathscr{U} = \{U_{i}\}_{i \in I}$, we say that $Z$ is a \textit{complex almost contact manifold} if it satisfies the following conditions:
\begin{enumerate}
\item On each $U_{i}$ there exist $1$-forms $u_{i}$, $v_{i} = u_{i} \circ \mathscr{J}$, vector fields $A_{i}, B_{i} = -\mathscr{J}A_{i}$, and $(1,1)$ tensor fields $G_{i}$, $H_{i} = G_{i} \circ \mathscr{J}$, such that
\begin{equation}
\label{relation1}
u_{i}(A_{i}) = 1, \ \ G_{i} \circ G_{i}= - {\rm{Id}} + u_{i} \otimes A_{i} + v_{i} \otimes B_{i}, \ \ G_{i} \circ \mathscr{J} = -\mathscr{J} \circ G_{i}, \ \ u_{i} \circ G_{i}= 0;
\end{equation}
\begin{equation}
\label{relation2}
H_{i}G_{i} = -G_{i}H_{i} = \mathscr{J} + u_{i} \otimes B_{i} - v_{i} \otimes A_{i}, \ \ v_{i} \circ G_{i} = v_{i} \circ H_{i} = u_{i} \circ H_{i} = 0,
\end{equation}
\begin{equation}
\label{relation3}
G_{i}A_{i} = G_{i}B_{i} = H_{i}A_{i} = H_{i}B_{i} = 0, \ \ \ v_{i}(A_{i}) = u_{i}(B_{i}) = 0.
\end{equation}
\item On $U_{i} \cap U_{j} \neq \emptyset$, we have $a,b \in C^{\infty}(U_{i} \cap U_{j})$, satisfying $a^{2} + b^{2} = 1$, such that
\begin{equation}
\label{trabsitioncontact}
\begin{cases} u_{j} = au_{i} - bv_{i}, \\ v_{j} = bu_{i} +av_{i}, \end{cases} \begin{cases} A_{j} = aA_{i} - bB_{i}, \\ B_{j} = bA_{i} + aB_{i}, \end{cases} \begin{cases} G_{j} = aG_{i} - bH_{i},\\ H_{j} = bG_{i} +aH_{i}. \end{cases}
\end{equation}

\end{enumerate}
We say that a complex almost contact manifold $Z$ is a \textit{complex almost contact metric manifold}, if, additionally, it admits a Hermirtian metric $g$ which satisfies 
\begin{equation}
u_{i}(X) = g(A_{i},X) \ \ \ {\text{and}} \ \ \ g(G_{i}X,Y) = -g(X,G_{i}Y),
\end{equation}
for any vector fields $X$ and $Y$. In the setting of complex contact manifolds, we have the following result:

\begin{theorem}[Ishihara \& Konishi, \cite{IshiharaKonishi}] 
\label{IshiharaKonishi}
Let $(Z, \mathscr{J},\theta)$ be a complex contact manifold of complex dimension $2n+1$. Then $Z$ admits a complex almost contact metric structure.
\end{theorem}

In what follows, we give a brief sketch of how such a complex almost contact metric structure provided by the above theorem can be obtained from the underlying complex contact structure, the details and complete proofs omitted here can be found in \cite{IshiharaKonishi}, \cite{Shibuya} and \cite{Foreman}.

Under the hypothesis of Theorem \ref{IshiharaKonishi}, and keeping the notation of the previous section, consider the Hermitian line bundle $(L,\langle \cdot\ , \cdot \rangle_{L})$, such that $L = E^{-1}$. Since $L$ is determined by the cocycles $g_{ij} = f_{ij}^{-1}$, from Eq. \ref{localhermitian} we obtain
\begin{center}
$h_{i} = |f_{ij}|^{2}h_{j}.$
\end{center}
Thus, by setting\footnote{Observe that, denoting $\tau_{i} = \sqrt{h_{i}}$, we have $\frac{f_{ij}}{|f_{ij}|} = \tau_{i}^{-1}f_{ij}\tau_{j}$, cf. \cite[Eq. 2.2]{IshiharaKonishi}.} $\varpi_{i} := \frac{\theta_{i}}{\sqrt{h_{i}}}$, on each $U_{i} \in \mathscr{U}$, we have
\begin{center}
$\varpi_{i} = \frac{f_{ij}}{|f_{ij}|}\varpi_{j} = \frac{1}{t_{ij}} \varpi_{j} \Longrightarrow \varpi_{j} = t_{ij}\varpi_{i} \ \ \ \text{on}  \ \ \ U_{i}\cap U_{j} \neq \emptyset,$
\end{center}
recall that $t_{ij} = \frac{g_{ij}}{|g_{ij}|} = \cos(\psi_{ij}) - \sqrt{-1}\sin(\psi_{ij})$, see Remark \ref{transitionangle}. 

\begin{remark}[Normalized contact structure]
\label{Hermitiandependence}
Notice that $\theta_{i} \wedge \big (d\theta_{i} \big)^{n} \neq 0 \Longrightarrow \varpi_{i} \wedge \big (d\varpi_{i} \big)^{n} \neq 0$, on every $U_{i} \in \mathscr{U}$, so the set $\{\varpi_{i}\}$ described above defines a {\textit{normalized contact structure}} on $Z$, see for instance \cite{Foreman}. By definition, we have that the normalized contact structure $\{\varpi_{i}\}$ depends on the Hermitian structure $\langle \cdot\ , \cdot \rangle_{L}$ on $L$.
\end{remark}

Denoting $\varpi_{i} = u_{i} - \sqrt{-1}v_{i}$, we obtain 
\begin{equation}
\label{realcontact}
u_{i} = \frac{\mathfrak{Re}(\theta_{i})}{\sqrt{h_{i}}} \ \ \text{and} \ \ v_{i} = - \frac{\mathfrak{Im}(\theta_{i})}{\sqrt{h_{i}}},
\end{equation}
notice that $v_{i} = u_{i} \circ \mathscr{J}$. Moreover, since $\varpi_{j} = t_{ij}\varpi_{i}$ on $U_{i} \cap U_{j} \neq \emptyset$, it follows that 
\begin{equation}
u_{j} = \cos(\psi_{ij})u_{i} - \sin(\psi_{ij})v_{i}, \ \ \ v_{j} = \sin(\psi_{ij})u_{i} + \cos(\psi_{ij})v_{i},
\end{equation}
see Eq. \ref{polarcocylce}. Now, we consider the collection of 1-forms $\sigma_{i} \in \Omega^{1}(U_{i})$, $U_{i} \in \mathscr{U}$, defined by
\begin{equation}
\label{gaugefield}
\sqrt{-1}\sigma_{i} := \frac{1}{2}(\partial - \overline{\partial}) \log(h_{i}).
\end{equation}
It is straightforward to verify that 
\begin{equation}
\label{gaugetransform}
\sqrt{-1}\sigma_{j} = \sqrt{-1}\sigma_{i} - d\log \Big ( \frac{f_{ij}}{|f_{ij}|}\Big ) = \sqrt{-1}\sigma_{i} + \frac{dt_{ij}}{t_{ij}} = \sqrt{-1}\big (\sigma_{i} - d\psi_{ij} \big ),
\end{equation}
on $U_{i}\cap U_{j} \neq \emptyset$, see Remark \ref{transitionangle}. From these local data, we define on each $U_{i} \in \mathscr{U}$ a complex-valued $2$-form $\Omega_{i}$ by
\begin{equation}
\Omega_{i} := d\varpi_{i} - \sqrt{-1}\sigma_{i} \wedge \varpi_{i}.
\end{equation}
The complex valued 2-forms defined above satisfies the following relation on $U_{i}\cap U_{j} \neq \emptyset$
\begin{center}
$\Omega_{i} = \frac{f_{ij}}{|f_{ij}|} \Omega_{j} \Longleftrightarrow  \Omega_{j} = t_{ij} \Omega_{i}.$
\end{center}
Moreover, since $ \varpi_{i} \wedge \big (d\varpi_{i} \big)^{n} \neq 0$, on $U_{i} \in \mathscr{U}$, we obtain
\begin{equation}
\varpi_{i} \wedge \Omega_{i}^{n} \neq 0 \ \ \ {\text{on}} \ \ \ U_{i} \in \mathscr{U}.
\end{equation}
Hence, each $\Omega_{i}$ is a complex-valued $2$-form of rank $2n$ on $U_{i}$. 
\begin{remark}[$\Omega$-structure]
It is worth to observe that 
\begin{equation}
d \varpi_{i} = \partial \varpi - \frac{1}{2}\overline{\partial}\log(h_{i}) \wedge \varpi_{i} \Longrightarrow \Omega_{i} = \partial \varpi_{i} - \frac{1}{2}\partial\log(h_{i}) \wedge \varpi_{i},
\end{equation}
so we have $\Omega_{i} \in \Omega^{2,0}(U_{i})$, for every $U_{i} \in \mathscr{U}$. Hence, the collection of (2,0)-forms $\{\Omega_{i}\}$ defines a $\Omega$-structure in the sense of Shibuya \cite{Shibuya}.
\end{remark}

The collection of (2,0)-forms $\{\Omega_{i}\}$ described above gives rise to skew-symmetric local tensor fields $\hat{G}_{i}$ and $\hat{H}_{i}$, defined by
\begin{equation}
\label{locendomorphisms}
\hat{G}_{i} := \mathfrak{Re}(\Omega_{i}) = du_{i} - \sigma_{i} \wedge v_{i} \ \ \ {\text{and}} \ \ \ \hat{H}_{i} := - \mathfrak{Im}(\Omega_{i}) = dv_{i} + \sigma_{i} \wedge u_{i}.
\end{equation}
It is straightforward to verify that
\begin{equation}
\label{localendomorphism}
\Omega_{j} = t_{ij} \Omega_{i} \Longrightarrow \begin{cases}
\hat{G}_{j} = \cos(\psi_{ij})\hat{G}_{i} - \sin(\psi_{ij})\hat{H}_{i}, \\ 
\hat{H}_{j} = \sin(\psi_{ij})\hat{G}_{i} + \cos(\psi_{ij})\hat{H}_{i},
    \                  
  \end{cases} 
\end{equation}
on $U_{i}\cap U_{j} \neq \emptyset$. Also, since $\Omega_{i} \in \Omega^{2,0}(U_{i})$, $\forall U_{i} \in \mathscr{U}$, it can be shown that 
\begin{equation}
\label{relationendomorphism}
\hat{H}_{i}(X,Y) = \hat{G}_{i}(\mathscr{J}X,Y) \ \ \ {\text{and}} \ \ \ \hat{G}_{i}(X,Y) = -\hat{H}_{i}(\mathscr{J}X,Y),
\end{equation}
for any vector fields $X$ and $Y$. Now, if we put for any $x \in U_{i} (\in \mathscr{U})$
\begin{center}
$\mathscr{V}_{i}(x) = \big \{ X \in T_{x}Z \ \big | \ \hat{G}_{i}(X,Y) = 0, \ \forall Y \in T_{x}Z \big \},$
\end{center}
we get on $U_{i}$ a local distribution $\mathscr{V}_{i} \colon x \mapsto \mathscr{V}_{i}(x)$. From Eq. \ref{localendomorphism} and Eq. \ref{relationendomorphism}, it follows that $\mathscr{V}_{i} = \mathscr{V}_{j}$ on $U_{i}\cap U_{j} \neq \emptyset$. Thus, we have a globally well-defined smooth distribution $\mathscr{V}$, which is of real dimension $2$. From this, it can be shown that there exist local defined smooth vector fields $A_{i},B_{i}$ on each $U_{i} \in \mathscr{U}$ which generates $\mathscr{V}_{i}$, satisfying 
\begin{equation}
u_{i}(A_{i}) = v_{i}(B_{i}) = 1 \ \ \ {\text{and}} \ \ \ u_{i}(B_{i}) = v_{i}(A_{i}) = 0.
\end{equation}
see for instance \cite[Lemma 3.1]{IshiharaKonishi}. Moreover, we have on $U_{i}\cap U_{j} \neq \emptyset$ the following relations
\begin{equation}
\label{transitionverctors}
A_{j} = \cos(\psi_{ij})A_{i} - \sin(\psi_{ij})B_{i}, \ \ \ \ B_{j} = \sin(\psi_{ij})A_{i} + \cos(\psi_{ij})B_{i}.
\end{equation}
\begin{remark}
If we consider $\mathscr{H}_{i} = \ker(\varpi_{i})$, on each $U_{i} \in \mathscr{U}$, from the definition of $\varpi_{i}$, we recover the distribution $\mathscr{H}$ which corresponds to $\mathscr{H}^{1,0} = \ker(\theta)$ under the identification $T^{1,0}Z  \cong (TZ,\mathscr{J})$. Moreover, we have the Whitney sum 
\begin{equation}
\label{spliting}
TZ \cong \mathscr{H} \oplus \mathscr{V}.
\end{equation}
Further, since $A_{j} + \sqrt{-1}B_{j} = {\rm{e}}^{\sqrt{-1}\psi_{ij}}(A_{i} + \sqrt{-1}B_{i})$, on $U_{i} \cap U_{j} \neq \emptyset$ (cf. Eq. \ref{polarcocylce}), by complexifying the fibers of $\mathscr{V}$, we obtain a $C^{\infty}$-isomorphism of complex line bunles $\mathscr{V} \cong E$ , so $TZ \cong \mathscr{H} \oplus E$
\end{remark}

Now, using the skew-symmetric local tensor fields $\hat{G}_{i}$, one can construct a Hermitian metric $g_{Z}$ on $Z$, satisfying
\begin{equation}
\label{compatiblemetric}
\hat{G}_{i}(X,Y) = g_{Z}(G_{i}X,Y) \ \ \ {\text{and}} \ \ \ u_{i}(X) = g_{Z}(A_{i},X)
\end{equation}
for any vector fields $X$ and $Y$, see for instance \cite{IshiharaKonishi}. The local $1$-forms $u_{i}$, $v_{i} = u_{i} \circ \mathscr{J}$, vector fields $A_{i}, B_{i} = -\mathscr{J}A_{i}$, and $(1,1)$ tensor fields $G_{i}$, $H_{i} = G_{i} \circ \mathscr{J}$, satisfy Eq. \ref{relation1}, Eq. \ref{relation2} and Eq. \ref{relation3}. Also, we have from Eq. \ref{localendomorphism}, and Eq. \ref{compatiblemetric}, that 
\begin{equation}
\label{overlapendomorphism}
G_{j} = \cos(\psi_{ij})G_{i} - \sin(\psi_{ij})H_{i}, \ \ \ H_{j} = \sin(\psi_{ij})G_{i} + \cos(\psi_{ij})H_{i},
\end{equation}
on $U_{i}\cap U_{j} \neq \emptyset$. From this, we obtain a complex almost contact metric structure on $(Z, \mathscr{J},\theta)$ completely determined by $\theta  \in H^{0}(Z,  \Omega_{Z}^{1}\otimes E)$.
\begin{remark}
Although the definition of $(u_{i},v_{i},A_{i},B_{i},G_{i},H_{i})$ depends on the choice made of some Hermitian structure $\langle \cdot\ , \cdot \rangle_{L}$ on $L$ (see Remark \ref{Hermitiandependence}), we have that the splitting describe in Eq. \ref{spliting} depends only on $\theta  \in H^{0}(Z,  \Omega_{Z}^{1}\otimes E)$. 
\end{remark}

\begin{remark}
\label{localKobayashi}
As we shall see afterwards, the locally defined structure tensors $(u_{i},v_{i},A_{i},B_{i},G_{i},H_{i})$ can be used to construct globally defined tensor fields on the manifold underlying the total space of the $\rm{U}(1)$-principal bundle $Q(L)$. An example of this procedure implicitly appears in Kobayashi's construction (Eq. \ref{Kobayashicontact}). In fact, the local expression of Kobayashi's real contact structure $\eta \in \Omega^{1}(Q(L))$, given in Eq. \ref{localcontactKobayashi}, can be rewritten using Eq. \ref{realcontact} as 
\begin{equation}
\label{localcontact}
\eta = \cos(\phi_{i}) \pi_{Q}^{\ast} (u_{i}) + \sin(\phi_{i})\pi_{Q}^{\ast}(v_{i}).
\end{equation}
Hence, one can realize Kobayashi's real contact structure $\eta$ as a globally defined 1-form obtained by gluing locally defined 1-forms as described on the right-hand side of Eq. \ref{localcontact}. In order to prove our main result, we will show that the locally defined structure tensors $(u_{i},v_{i},A_{i},B_{i},G_{i},H_{i})$ can be used to construct an almost contact structure on $Q(L)$, which is compatible, in a suitable sense, with Hatakeyama's almost contact structure \cite{HATAKEYAMA}.
\end{remark}

\section{Generalities on Almost contact geometry}
\label{Sec3}

In this section, we review some basic generalities on almost contact geometry, contact metric geometry, and almost 3-contact geometry. Further, we also introduce some notations to be used in the next sections. 

\subsection{Almost contact manifolds} Let us recall some basic facts and generalities on almost contact geometry. Our approach is based on \cite{SASAKIHARAKEYMAALMOST}, \cite{Blair}. 

An \textit{almost contact manifold} is a real $(2n+1)$-dimensional smooth manifold $M$ endowed with structure tensors $(\Phi, \xi,\eta)$, such that $\Phi \in {\text{End}}(TM)$, $\xi \in \mathfrak{X}(M)$, and $\eta \in \Omega^{1}(M)$, satisfying
\begin{equation}
\label{almostcontact}
 \Phi \circ \Phi = - {\rm{Id}} + \eta \otimes \xi, \ \ \eta(\xi) = 1.    
\end{equation}

An almost contact structure $(\Phi, \xi,\eta)$ is said to be \textit{normal} if it satisfies 
 \begin{equation}
\label{normality}
\big [ \Phi,\Phi \big ] + 2d\eta \otimes \xi = 0.
\end{equation}
A smooth manifold $M$ endowed with a normal almost contact structure $(\Phi, \xi,\eta)$ is called \textit{normal almost contact manifold}. An alternative way to characterize the normality condition for an almost contact structure is provided as follows: Given an almost contact manifold $M$ with structure tensors $(\Phi,\xi,\eta)$, we can consider the manifold defined by $M \times \mathbbm{R}$. Denoting by $t$ the coordinate on $\mathbbm{R}$, we can define an almost complex structure $\mathbbm{I}$ on $M \times \mathbbm{R}$ such that
\begin{equation}
\label{complexcone}
\mathbbm{I}(X) : = \Phi(X) - \eta(X)\frac{d}{dt}, \ \ \ \ \mathbbm{I}\Big ( \frac{d}{dt} \Big) := \xi.
\end{equation}
for all $X \in \mathfrak{X}(M)$. By following \cite{SASAKIHARAKEYMAALMOST}, we can show that 

\begin{equation}
\label{Normalintegrable}
\big [ \Phi,\Phi \big ] + 2d\eta \otimes \xi = 0 \Longleftrightarrow \big [ \mathbbm{I},\mathbbm{I} \big ] = 0.
\end{equation}
Thus, we have the normality condition for $(\Phi,\xi,\eta)$ equivalent to the integrability for the almost complex structure $\mathbbm{I}$ defined in \ref{complexcone}.

A Riemannian metric $g$ on an almost contact manifold $M$ is said to be compatible with its almost contact structure $(\Phi, \xi,\eta)$ if
\begin{equation}
g(\xi,X) = \eta(X), \ \ \ g(\Phi(X),\Phi(Y)) = g(X,Y) - \eta(X)\eta(Y),
\end{equation}
for any vector fields $X$ and $Y$. An almost contact manifold with a compatible Riemannian metric is called \textit{normal almost contact metric manifold}. In \cite{Sasaki}, it was shown that every almost contact manifold admits a compatible Riemannian metric.

An important result which allows us to get a huge class of examples of normal almost contact manifolds is the following theorem:

\begin{theorem}[Hatakeyama, \cite{HATAKEYAMA}; Morimoto, \cite{MORIMOTO}]
\label{almostcircle}
Let $Q$ be the total space of a $\rm{U}(1)$-principal bundle over a complex manifold $(M,J)$. Suppose we have a connection $1$-form $\sqrt{-1}\eta$ on $Q$, such that $d\eta = \pi^{\ast}(\omega)$, here $\pi$ denotes the projection of $Q$ onto $M$, and suppose that $\omega$ is a $2$-form on $M$ satisfying 

\begin{center}

$\omega(JX,JY) = \omega(X,Y),$

\end{center}
for $X,Y \in \mathfrak{X}(M)$. Then, we can define a $(1,1)$-tensor field $\Phi$ on $Q$, and a vector field $\xi$ on $Q$, such that $(\Phi,\xi,\eta)$ is a normal almost contact structure on $Q$.
\end{theorem}

The description of the normal almost contact structure given in the above theorem can be easily described as follows: Consider $\xi = \frac{\partial}{\partial \phi} \in \mathfrak{X}(Q)$ as being the vector field defined by the infinitesimal action of $\mathfrak{u}(1)$ on $Q$ and let $\sqrt{-1}\eta \in \Omega^{1}(Q;\mathfrak{u}(1))$ be a connection $1$-form, such that $d\eta = \pi^{\ast}(\omega)$. Without loss of generality, we can suppose that $\eta(\xi) = 1$. Now, we define $\Phi \in {\text{End}}(TQ)$ by setting
\begin{equation}
\label{IKendomorphism}
\Phi(X) := (J\pi_{\ast}(X))^{\#},  \ \ \ \ \forall X \in TQ.
\end{equation}
Here we denote by $(J\pi_{\ast}(X))^{\#}$ the horizontal lift of $J\pi_{\ast}(X)$ relative to the connection $\sqrt{-1}\eta \in \Omega^{1}(Q;\mathfrak{u}(1))$. A straightforward computation shows that $(\Phi, \xi, \eta)$ defines an almost contact structure, see for instance \cite[Theorem 1]{HATAKEYAMA}. For the normality condition, we just need to check that

\begin{center}

$\big [ J,J \big ] \equiv 0$  \ \ and \ \ $\omega \in \Omega^{1,1}(M) \implies \big [ \Phi ,\Phi \big ] + 2d\eta \otimes \xi = 0,$

\end{center}
the details of the implication above can be found in \cite[Theorem 2]{HATAKEYAMA}, \cite[Theorem 6]{MORIMOTO}.

\begin{remark}[IK-connection]
\label{IKconnection}
Let $(Z, \mathscr{J},\theta)$ be a complex contact manifold of complex dimension $2n+1$. As we have seen, there exists a system of locally defined $1$-forms $\mathscr{A} = \{(\sigma_{i},U_{i})\}$, which satisfies 
\begin{center}
$\sqrt{-1}\sigma_{j} = \sqrt{-1}\big (\sigma_{i} - d\psi_{ij} \big )  \ \ \text{on}  \ \ \ U_{i}\cap U_{j} \neq \emptyset,$
\end{center}
see Eq. \ref{gaugefield} and Eq. \ref{gaugetransform}. From the relation above, we have a connection $1$-form $\sqrt{-1}\eta_{1} \in \Omega^{1}(Q(L);\mathfrak{u}(1))$, such that 
\begin{equation}
\label{localIK}
\eta_{1} = \pi_{Q}^{\ast}(\sigma_{i}) + d\phi_{i} \ \ \ \text{on}  \ \ \ Q(L)|_{U_{i}},
\end{equation}
notice that $d\phi_{j} = d\phi_{i} + \pi_{Q}^{\ast}(d\psi_{ij})$, on $U_{i}\cap U_{j} \neq \emptyset$, see Remark \ref{transitionangle}. The connection $\sqrt{-1}\eta_{1}$ is called \textit{Ishihara-Konishi connection} (IK-connection). Further, we have $d\eta_{1} = \pi_{Q}^{\ast}(\omega)$, where $\omega \in \Omega^{1,1}(Z)$, such that 
\begin{equation}
\label{curvature}
\omega = \sqrt{-1} \partial \overline{\partial} \log(h_{i}) \ \ \ \text{on} \ \ \ U_{i} \in \mathscr{U}.
\end{equation}
Therefore, from Theorem \ref{almostcircle}, we obtain a normal almost contact structure $(\Phi_{1},\xi_{1},\eta_{1})$ on $Q(L)$. 
\end{remark}

\subsection{Contact metric structures} A {\textit{contact manifold}} is a pair $(M,\eta)$, where $M$ is manifold of real dimension $2n+1$, and $\eta$ is a 1-form ({\textit{contact structure}}) which satisfies $\eta \wedge (d\eta)^{n} \neq 0$. Associated to the contact structure $\eta$, we have a smooth vector field $\xi \in \mathfrak{X}(M)$, called \textit{characteristic (or Reeb) vector field}, which satisfies 
\begin{center}
$\xi \lrcorner d\eta = 0$ \ \ \ and \ \ \ $\xi \lrcorner \eta = 1.$
\end{center}
Given a contact manifold $(M,\eta)$, with characteristic vector field $\xi \in \mathfrak{X}(M)$, we say that a Riemannian metric $g$ is an {\textit{associated metric}} if the following properties are satisfied:
\begin{enumerate}

\item $\eta(X) = g(X,\xi)$;
\item There exists a (1,1)-tensor field $\Phi \in {\text{End}}(TM)$, satisfying 
\begin{equation}
\Phi \circ \Phi = - {\rm{Id}} + \eta \otimes \xi \ \ \ {\text{and}} \ \ \ d\eta = g(\Phi \otimes {\rm{Id}} ). 
\end{equation}
\end{enumerate}
A contact manifold with an associated metric is called {\textit{contact metric manifold}}. As one can see, every contact metric manifold is particularly an almost contact metric manifold. In fact, from  $(1)$ and $(2)$, we have
\begin{equation}
\label{compatible}
g(\Phi(X),\Phi(Y)) = d\eta(X,\Phi(Y)) = -d\eta(\Phi(Y),X) = -g(\Phi(\Phi(Y)),X) = g(X,Y) -\eta(X)\eta(Y).
\end{equation}
Given a contact metric manifold $(M,\eta)$, with associated metric $g$, if the underlying almost contact structure $(\Phi,\xi,\eta)$ is normal, then $M$ is called {\textit{Sasaki manifold}} \cite{Sasaki}.

\begin{remark}
\label{SEandKE}
Let $M$ be a smooth manifold with a Sasakian structure $(g,\Phi,\xi,\eta)$, such that $\dim_{\mathbbm{R}}(M) = 2n+1$. By setting $\mathscr{D} := \ker(\eta)$, it can be shown that $(\mathscr{D},\Phi|_{\mathscr{D}},g|_{\mathscr{D}})$ defines a K\"{a}hler foliation on $M$, see for instance \cite[Corollary 6.5.11]{Sasakigeometry}. Defining $g^{T} := g|_{\mathscr{D}}$, we have the following relations:
\begin{enumerate}
\item ${\text{Ric}}_{g}(X,Y) = {\text{Ric}}_{g^{T}}(X,Y) - 2g(X,Y)$, \ \ \ $\forall X,Y \in \mathscr{D}$,

\item ${\text{Ric}}_{g}(X,\xi) = 2n\eta(X)$, \ \ \ $\forall X \in TM$,
\end{enumerate}
see for instance \cite[Theorem 7.3.12]{Sasakigeometry}. Hence, if $g$ is a Einstein, we have ${\text{Scal}}_{g} = 2n(2n+1)$, and ${\text{Ric}}_{g^{T}} = 2(n+1)g^{T}$. In this last case, if the characteristic foliation $\mathscr{F}_{\xi}$ defined by the characteristic vector field $\xi \in \mathfrak{X}(M)$ is regular, then the transverse K\"{a}hler-Einstein structure $(\mathscr{D}, \Phi|_{\mathscr{D}},g^{T})$ pushes down to a K\"{a}hler-Einstein structure on the smooth manifold $N:= M/\mathscr{F}_{\xi}$ defined by the leaf space.
\end{remark}

\subsection{Almost contact 3-structures} An \textit{almost contact} 3-\textit{structure} on a smooth manifold $M$ is defined by three almost contact structures $(\Phi_{\alpha},\xi_{\alpha},\eta_{\alpha})$, $\alpha = 1,2,3$, satisfying
\begin{equation}
\label{R1}
\Phi_{\alpha} \circ \Phi_{\beta} -  \eta_{\beta} \otimes \xi_{\alpha}  = - \Phi_{\beta} \circ \Phi_{\alpha} +  \eta_{\alpha} \otimes \xi_{\beta}  = \Phi_{\gamma},
\end{equation}
\begin{equation}
\Phi_{\alpha}(\xi_{\beta}) = -\Phi_{\beta}(\xi_{\alpha}) = \xi_{\gamma}, \ \ \eta_{\alpha} \circ \Phi_{\beta} = - \eta_{\beta} \circ \Phi_{\alpha} = \eta_{\gamma},
\end{equation}
for any cyclic permutation $(\alpha,\beta,\gamma)$ of $(1,2,3)$. The notion of almost contact 3-structure was introduced by Kuo \cite{Kuo} and independently under the name almost coquaternion structure by Udriste \cite{Udriste}. A manifold endowed with an almost contact 3-structure is called {\textit{almost 3-contact manifold}}. For our purpose, it will be important to consider the following results:

\begin{theorem}[Kuo, \cite{Kuo}]
\label{Kuotheorem}
If a smooth manifold admits two almost contact structures $(\Phi_{\alpha},\xi_{\alpha},\eta_{\alpha})$, $\alpha = 1,2$, satisfying
\begin{equation}
\label{L1eq}
\Phi_{1}(\xi_{2}) = -\Phi_{2}(\xi_{1}), \ \ \eta_{1} \circ \Phi_{2} = - \eta_{2} \circ \Phi_{1}, \ \ \eta_{1}(\xi_{2}) = \eta_{2}(\xi_{1}) = 0,
\end{equation}
\begin{equation}
\label{L2eq}
\Phi_{1} \circ \Phi_{2} - \eta_{2} \otimes \xi_{1} = - \Phi_{2} \circ \Phi_{1} + \eta_{1} \otimes \xi_{2},
\end{equation}
then it admits an almost contact $3$-structure.
\end{theorem}
\begin{remark}
Under the hypothesis of the above theorem, the third almost contact structure can be obtained by setting
\begin{equation}
\label{thirdalmostcontact}
\xi_{3} = \Phi_{1}(\xi_{2}), \ \ \eta_{3} = \eta_{1} \circ \Phi_{2}, \ \ \Phi_{3} = \Phi_{1} \circ \Phi_{2} - \eta_{2} \otimes \xi_{1},
\end{equation}
\end{remark}
\begin{theorem}[Yano, Ishihara, Konishi, \cite{Yano}]
\label{3normality}
If, for an almost contact 3-structure $(\Phi_{\alpha},\xi_{\alpha},\eta_{\alpha})$, $\alpha = 1,2,3$, any two of almost contact structures $(\Phi_{\alpha},\xi_{\alpha},\eta_{\alpha})$ are normal, then the third is so.
\end{theorem}

\begin{remark}
\label{hypercomplex}
Given an almost contact $3$-structure $(\Phi_{\alpha},\xi_{\alpha},\eta_{\alpha})$, $\alpha = 1,2,3$, on a smooth manifold $M$, following Eq. \ref{complexcone}, we have three almost complex structures on $M \times \mathbbm{R}$, such that 
\begin{center}
$\displaystyle \mathbbm{I}_{\alpha}(X) : = \Phi_{\alpha}(X) - \eta_{\alpha}(X)\frac{d}{dt}, \ \ \ \ \mathbbm{I}_{\alpha}\Big ( \frac{d}{dt} \Big) := \xi_{\alpha}$, \ \ \ \ $(\alpha = 1,2,3)$
\end{center}
for every $X \in \mathfrak{X}(M)$. From Eq. \ref{R1}, it can be shown that 
\begin{equation}
\mathbbm{I}_{\alpha} \circ \mathbbm{I}_{\beta} = - \mathbbm{I}_{\beta} \circ \mathbbm{I}_{\alpha} = \mathbbm{I}_{\gamma}, 
\end{equation}
for any cyclic permutation $(\alpha,\beta,\gamma)$ of $(1,2,3)$. Therefore, we have that $(M \times \mathbbm{R}, \mathbbm{I}_{1},\mathbbm{I}_{2},\mathbbm{I}_{3})$ defines an {\textit{almost hypercomplex manifold}}. In the case that each one of the almost complex structures $\mathbbm{I}_{1},\mathbbm{I}_{2},\mathbbm{I}_{3}$, is integrable, we have that $(M \times \mathbbm{R}, \mathbbm{I}_{1},\mathbbm{I}_{2},\mathbbm{I}_{3})$ is a {\textit{hypercomplex manifold}}. Notice that, if any two of the almost contact structures $(\Phi_{\alpha},\xi_{\alpha},\eta_{\alpha})$, $\alpha = 1,2,3$, are normal, then from Theorem \ref{3normality}, and Eq. \ref{Normalintegrable}, we have that $(M \times \mathbbm{R}, \mathbbm{I}_{1},\mathbbm{I}_{2},\mathbbm{I}_{3})$ is hypercomplex.
\end{remark}

Given an almost contact $3$-structure $(\Phi_{\alpha},\xi_{\alpha},\eta_{\alpha})$, $\alpha = 1,2,3$, a Riemannian metric $g$ is said to be compatible with the almost contact $3$-structure if 
\begin{equation}
g(\xi_{\alpha},X) = \eta_{\alpha}(X), \ \ \ g(\Phi_{\alpha}(X),\Phi_{\alpha}(Y)) = g(X,Y) - \eta_{\alpha}(X)\eta_{\alpha}(Y), \ \ \ \ (\alpha = 1,2,3)
\end{equation}
for any vector fields $X$ and $Y$. An almost contact 3-structure with a compatible Riemannian metric is called {\textit{almost contact metric 3-structure}}. 

As in the case of almost contact structures, the following theorem ensures the existence of associated Riemannian metrics on almost 3-contact manifolds.

\begin{theorem}[Kuo, \cite{Kuo}]
Every manifold with an almost contact $3$-structure admits a compatible Riemannian metric.
\end{theorem}

\begin{remark}
Let $(\Phi_{\alpha},\xi_{\alpha},\eta_{\alpha})$, $\alpha = 1,2,3$, be an almost contact $3$-structure on a smooth manifold $M$, and consider $g_{M}$ as being the compatible Riemannian metric provided by the last theorem. As we have seen, an almost contact $3$-structure induces an almost hypercomplex structure $(\mathbbm{I}_{1},\mathbbm{I}_{2},\mathbbm{I}_{3})$ on the manifold on $\mathscr{C}(M) := M \times \mathbbm{R}$. By considering $g_{M}$, we can define a Riemannian metric $g_{\mathscr{C}}$ on $\mathscr{C}(M)$ such that 
\begin{equation}
\label{conemetric}
g_{\mathscr{C}} := g_{M} + dt \otimes dt.
\end{equation}
From the above metric, we have that $(g_{\mathscr{C}},\mathbbm{I}_{\alpha})$ defines an almost Hermitian structure on $\mathscr{C}(M)$, for $\alpha = 1,2,3$. Thus, in this case, we have that $(\mathscr{C}(M), g_{\mathscr{C}}, \mathbbm{I}_{1},\mathbbm{I}_{2},\mathbbm{I}_{3})$ is an {\textit{almost hyperhermitian manifold}}.
\end{remark}

\subsection{Contact metric 3-structures} A {\textit{contact metric 3-structure}} on a smooth manifold $M$ is defined by a set of three contact structures $\eta_{\alpha}$, $\alpha = 1,2,3$, with same associated metric $g$, such that the underlying almost contact structures $(\Phi_{\alpha},\xi_{\alpha},\eta_{\alpha})$, $\alpha = 1,2,3$, satisfy Eq. \ref{R1}. Notice that, in this case, the structure tensors $(\Phi_{\alpha},\xi_{\alpha},\eta_{\alpha})$, $\alpha = 1,2,3$, in fact define an almost contact 3-structure, see \cite[Theorem 2]{Kuo}. A contact metric $3$-structure is said to be {\textit{3-Sasaki}} if each one of the underlying almost contact structures $(\Phi_{\alpha},\xi_{\alpha},\eta_{\alpha})$, $\alpha = 1,2,3$, is a normal almost contact structure. A manifold which admits a 3-Sasaki structure is called {\textit{3-Sasakian manifold}}. The following result will be important for us later.

\begin{theorem}[Kuo, \cite{Kuo}] 
\label{2sasaki}
If, for an almost contact metric 3-structure $(g_{M},\Phi_{\alpha},\xi_{\alpha},\eta_{\alpha})$, $\alpha = 1,2,3$, any two of almost contact metric structures $(g_{M},\Phi_{\alpha},\xi_{\alpha},\eta_{\alpha})$ are Sasaki, then the third is so.
\end{theorem}

A remarkable result in the setting of contact metric 3-structures is given by the following theorem:

\begin{theorem}[Kashiwada, \cite{Kashiwada}] 
\label{Kashiwada}
A contact metric 3-structure is necessarily a 3-Sasaki structure.
\end{theorem}

\begin{remark}
\label{Hyperkahler3contact}
In order to fix some notations to be used latter, let us sketch the proof of the last theorem. Let $M$ be a manifold with a contact metric 3-structure $\eta_{\alpha}$, $\alpha = 1,2,3$, and associated metric $g_{M}$. Considering the underlying almost contact 3-structure $(\Phi_{\alpha},\xi_{\alpha},\eta_{\alpha})$, $\alpha = 1,2,3$, by definition of Sasakian manifolds, it is enough to show that $(\Phi_{\alpha},\xi_{\alpha},\eta_{\alpha})$ is normal, for $\alpha = 1,2,3$. From Remark \ref{hypercomplex}, we have an almost hypercomplex structure $(\mathbbm{I}_{1},\mathbbm{I}_{2},\mathbbm{I}_{3})$ on $\mathscr{C}(M) = M \times \mathbbm{R}$. Further, we can define a Riemannian metric metric $g_{\mathscr{C}}$ on $\mathscr{C}(M)$, where $g_{\mathscr{C}}$ is given as in Eq. \ref{conemetric}. From these, we have that $(g_{\mathscr{C}}, \mathbbm{I}_{1},\mathbbm{I}_{2},\mathbbm{I}_{3})$  defines an almost hyperhermitian structure on $\mathscr{C}(M)$. Since $g_{M}$ is also compatible with the almost contact structures $(\Phi_{\alpha},\xi_{\alpha},\eta_{\alpha})$, $\alpha = 1,2,3$, see Eq. \ref{compatible}, it follows that $\Theta_{\alpha} := g_{\mathscr{C}}(\mathbbm{I}_{\alpha} \otimes {\rm{Id}}) = d\eta_{\alpha} + dt \wedge \eta_{\alpha}$, thus
\begin{equation}
d\Theta_{\alpha} = \Theta_{\alpha} \wedge dt, \ \ \ \ \ \ \ \ \ (\alpha = 1,2,3).
\end{equation}
Therefore, by setting $g_{{\text{HK}}} = {\rm{e}}^{t}g_{\mathscr{C}}$, we have that $(g_{{\text{HK}}},\mathbbm{I}_{1},\mathbbm{I}_{2},\mathbbm{I}_{3})$ also defines an almost hyperhermitian structure on $\mathscr{C}(M)$. Moreover, since 
\begin{equation}
\omega_{\alpha} := g_{{\text{HK}}}(\mathbbm{I}_{\alpha} \otimes {\rm{Id}}) = {\rm{e}}^{t}\Theta_{\alpha} = d({\rm{e}}^{t} \eta_{\alpha}), \ \ \ \ \ \ \ \ \ (\alpha = 1,2,3)
\end{equation}
it follows from \cite[Lemma 6.8]{Hitchin} that $(g_{{\text{HK}}},\mathbbm{I}_{1},\mathbbm{I}_{2},\mathbbm{I}_{3})$ is in fact a \textit{hyperk\"{a}hler structure} on $\mathscr{C}(M)$. In particular, we have that $\mathbbm{I}_{\alpha}$, $\alpha = 1,2,3$, are integrable, which in turn implies that $(\Phi_{\alpha},\xi_{\alpha},\eta_{\alpha})$, $\alpha = 1,2,3$, are normal. 
\end{remark}

Now, we consider the following result:

\begin{theorem}[Kashiwada, \cite{Kashiwada2}; cf. Boyer, Galicki, \cite{3Einstein}]
\label{Kashiwada2}
Let $M$ be a smooth manifold with a 3-Sasakian structure $(g_{M},\Phi_{\alpha},\xi_{\alpha},\eta_{\alpha})$, $\alpha = 1,2,3$. Then $(M,g_{M})$ is an Einstein space with positive scalar curvature.
\end{theorem}

Given a manifold $M$ with a 3-Sasaki structure $(g_{M},\Phi_{\alpha},\xi_{\alpha},\eta_{\alpha})$, $\alpha = 1,2,3$, if one supposes that $(\Phi_{1},\xi_{1},\eta_{1})$ is a regular Sasaki structure, it follows from Remark \ref{SEandKE} that, particularly, the transverse structure $(\mathscr{D} :=\ker(\eta_{1}), \Phi_{1}|_{\mathscr{D}_{1}},g^{T})$ pushes down to a K\"{a}hler-Einstein structure $(g,J)$ on the smooth manifold $N := M/\mathscr{F}_{\xi_{1}}$. Actually, we have even more:

\begin{theorem}[Ishihara \& Konishi, \cite{IshiharaKonishi1} ]
Let $M$ be a smooth manifold with a 3-Sasaki structure $(g_{M},\Phi_{\alpha},\xi_{\alpha},\eta_{\alpha})$,  $\alpha = 1,2,3$, and suppose that one of the Sasaki structures, say $(\Phi_{1},\xi_{1},\eta_{1})$, is a regular Sasaki structure. Then the leaf space $N:= M/\mathscr{F}_{\xi_{1}}$ admits a complex almost contact metric structure. Moreover, the complex almost contact metric structure on the orbit space is a K\"{a}hler-Einstein structure of positive scalar curvature. 
\end{theorem}
\begin{remark}
In view of the result above, the main theorem of this work can be thought of as a converse construction which allows us to go from complex contact geometry to almost 3-contact metric geometry from complex contact. 
\end{remark}
\section{Proof of main results}
\label{Sec4}

In this section, we provide a complete proof for our main result and its corollaries. For the sake of simplicity, we shall restate the results presented in the introduction.

\subsection{Proof of Theorem \ref{T1}} Let $(Z, \mathscr{J},\theta)$ be a complex contact manifold, and consider its complex almost contact structure $(u_{i},v_{i},A_{i},B_{i},G_{i},H_{i})$ provided by Theorem \ref{IshiharaKonishi}. In what follows, we consider the notation introduced in Section \ref{Cplxalmostsection}. Also, we shall denote by $X^{\#}$ the horizontal lift relative to the IK-connection (Remark \ref{IKconnection}), for any vector field $X$ on $Z$. In order to prove the main theorem (Theorem \ref{T1}), we first prove two technical lemmas:
\begin{lemma-non}
\label{localtensors}
The locally defined tensor fields:
\begin{enumerate}
\item $\Psi_{i}(X) := \cos(\phi_{i})(G_{i}\pi_{Q \ast}(X))^{\#} + \sin(\phi_{i})(H_{i}\pi_{Q \ast}(X))^{\#}$, $\forall X \in \mathfrak{X}(\pi_{Q}^{-1}(U_{i}))$,

\item $ \Xi_{i} := \cos(\phi_{i})A_{i}^{\#} + \sin(\phi_{i})B_{i}^{\#} \in \mathfrak{X}(\pi_{Q}^{-1}(U_{i}))$,
\end{enumerate}
satisfy $\Psi_{i} = \Psi_{j}$, and $\Xi_{i} = \Xi_{j}$, on $\pi_{Q}^{-1}(U_{i} \cap U_{j}) \neq \emptyset$.
\end{lemma-non}

\begin{proof}
Given any $X \in \mathfrak{X}(\pi_{Q}^{-1}(U_{i} \cap U_{j}))$, from Eq. \ref{overlapendomorphism}, it follows that 
\begin{equation}
\label{Eqloc11}
\cos(\phi_{j})(G_{j}\pi_{Q \ast}(X))^{\#} = \cos(\phi_{j})\cos(\psi_{ij} \circ \pi_{Q})(G_{i}\pi_{Q \ast}(X))^{\#} - \cos(\phi_{j})\sin(\psi_{ij} \circ \pi_{Q})(H_{i}\pi_{Q \ast}(X))^{\#}, 
\end{equation}
and
\begin{equation}
\label{Eqloc112}
\sin(\phi_{j})(H_{j}\pi_{Q \ast}(X))^{\#} = \sin(\phi_{j})\sin(\psi_{ij} \circ \pi_{Q})(G_{i}\pi_{Q \ast}(X))^{\#} + \sin(\phi_{j})\cos(\psi_{ij} \circ \pi_{Q})(H_{i}\pi_{Q \ast}(X))^{\#}.
\end{equation}
Hence, summing Eq. \ref{Eqloc11} and Eq. \ref{Eqloc112}, and rearranging the result, we obtain
\begin{center}
$\Psi_{j}(X) = \cos\big (\phi_{j} - \psi_{ij} \circ \pi_{Q}\big )(G_{i}\pi_{Q \ast}(X))^{\#} + \sin\big (\phi_{j} - \psi_{ij} \circ \pi_{Q}\big )(H_{i}\pi_{Q \ast}(X))^{\#} $.
\end{center}
Now, since $\phi_{j} = \phi_{i} + \psi_{ij} \circ \pi_{Q} + 2\pi k$, 
on $ \pi_{Q}^{-1}(U_{i} \cap U_{j})$, with $k \in \mathbbm{Z}$ (see Remark \ref{transitionangle}), we have $\Psi_{i}(X) = \Psi_{j}(X)$, for all $X \in \mathfrak{X}(\pi_{Q}^{-1}(U_{i} \cap U_{j}))$. From Eq. \ref{transitionverctors}, and a similar computation as above, we can also verify that $\Xi_{i} = \Xi_{j}$, on $\pi_{Q}^{-1}(U_{i} \cap U_{j}) \neq \emptyset$. 
\end{proof}

The result above allows us to define a smooth $(1,1)$-tensor field $\Psi \in {\text{End}}(TQ(L))$, and a smooth vector field $\Xi \in \mathfrak{X}(Q(L))$, by gluing the local data $\{\Psi_{i}\}$ and $\{\Xi_{i}\}$. By considering the Kobayashi's contact structure $\eta \in \Omega^{1}(Q(L))$, from Eq. \ref{localcontact}, and the local description of $(\Psi,\Xi)$, we obtain
\begin{equation}
\label{firstrelation}
\Psi(\Xi) = 0, \ \ \ \eta(\Xi) = 1 \ \ \ {\text{and}} \ \ \ \eta \circ \Psi = 0.
\end{equation}
Now, consider the (normal) almost contact structure $(\Phi_{1},\xi_{1},\eta_{1})$ on $Q(L)$ provided by the IK-connection (Remark \ref{IKconnection}). Denoting $\eta_{2} := \eta$, and $\xi_{2} := \Xi$, we define $\Phi_{2} \in {\text{End}}(TQ(L))$, such that
\begin{equation}
\label{secondendomorphism}
\Phi_{2} := \Psi - \eta_{1} \otimes \Phi_{1}(\xi_{2}) - (\eta_{2} \circ \Phi_{1}) \otimes \xi_{1}.
\end{equation}
Notice that, since $v_{i} = u_{i} \circ \mathscr{J}$ and $B_{i} = - \mathscr{J}A_{i}$, from Eq. \ref{IKendomorphism}, locally, we have
\begin{equation}
\label{LocEq1}
\Phi_{1}(\xi_{2}) = \sin(\phi_{i})A_{i}^{\#} - \cos(\phi_{i})B_{i}^{\#} \ \ \ \ \text{and} \ \ \ \ \eta_{2} \circ \Phi_{1} = - \sin(\phi_{i})\pi_{Q}^{\ast}(u_{i}) + \cos(\phi_{i})\pi_{Q}^{\ast}(v_{i}).
\end{equation}
\begin{remark}
\label{firstproperties}
From the definition of $(\Phi_{\alpha},\xi_{\alpha},\eta_{\alpha})$, $\alpha = 1,2,$ and Eq. \ref{LocEq1}, the following properties can be easily verified:
\begin{equation}
\eta_{2}(\Phi_{1}(\xi_{2})) = \eta_{2}(\xi_{1}) = \eta_{1}(\xi_{2}) = 0
\end{equation}
\begin{equation}
\Phi_{2}(\xi_{1}) = -\Phi_{1}(\xi_{2}), \ \ \Phi_{2}(\xi_{2}) = 0, \ \ \Phi_{2}(\Phi_{1}(\xi_{2})) = \xi_{1}
\end{equation}
\end{remark}
Now, from the above facts, we have the following lemma:
\begin{lemma-non}
\label{Kobayashialmostcontact}
$(\Phi_{2},\xi_{2},\eta_{2})$ defines an almost contact structure on $Q(L)$. 
\end{lemma-non}

\begin{proof}
We already have seen that $\eta_{2}(\xi_{2}) = 1$, see Eq. \ref{firstrelation}. Thus we just need to verify that
\begin{center}
$\Phi_{2} \circ \Phi_{2} = - {\rm{Id}} + \eta_{2} \otimes \xi_{2}.$
\end{center}
In order to prove the above equation, consider the decomposition induced by the IK-connection on $TQ(L)$, i.e.,
\begin{equation}
TQ(L) = {\text{Hor}}(TQ(L) ) \oplus {\text{Vert}}(TQ(L)),
\end{equation}
such that ${\text{Hor}}(TQ(L) ) = \ker(\eta_{1})$. Given $X \in TQ(L)$, we have
\begin{equation}
X = X^{\text{hor}} + \eta_{1}(X)\xi_{1},
\end{equation}
where $X^{\text{hor}}  = X-\eta_{1}(X)\xi_{1}$. From the above decomposition, we obtain
\begin{center}
$(\Phi_{2} \circ \Phi_{2})(X) = (\Phi_{2} \circ \Phi_{2})( X^{\text{hor}} ) + \eta_{1}(X)(\Phi_{2} \circ \Phi_{2})(\xi_{1})$
\end{center}
Since $\eta_{2}(\xi_{1}) = 0$, $\Phi_{2}(\xi_{1}) = -\Phi_{1}(\xi_{2})$, and $\Phi_{2}(\Phi_{1}(\xi_{2})) = \xi_{1}$, we have $(\Phi_{2} \circ \Phi_{2})(\xi_{1}) = -\xi_{1} + \eta_{2}(\xi_{1})\xi_{2}$, so
\begin{equation}
\label{verticalexpression}
(\Phi_{2} \circ \Phi_{2})(X) = (\Phi_{2} \circ \Phi_{2})( X^{\text{hor}} ) -\eta_{1}(X)\xi_{1} + \eta_{1}(X)\eta_{2}(\xi_{1})\xi_{2}.
\end{equation}
Now, for $X^{\text{hor}} \in {\text{Hor}}(TQ(L) )$, we have
\begin{equation}
\label{LocEq2}
\Phi_{2}(X^{\text{hor}}) = \Psi(X^{\text{hor}}) - \eta_{2}(\Phi_{1}(X^{\text{hor}}))\xi_{1} \Longrightarrow (\Phi_{2} \circ \Phi_{2})(X^{\text{hor}}) = \Phi_{2}(\Psi(X^{\text{hor}})) - \eta_{2}(\Phi_{1}(X^{\text{hor}}))\Phi_{2}(\xi_{1}).
\end{equation}
Since $v_{i} \circ G_{i} = v_{i} \circ H_{i} = u_{i} \circ H_{i} = 0$, it follows that $\eta_{2}(\Phi_{1}(\Psi(X^{\text{hor}}))) = 0$, and from the definition of $\Psi$, we have $\eta_{1}(\Psi(Y)) = 0$. Hence, it follows that $\Phi_{2}(\Psi(X^{\text{hor}})) = (\Psi \circ \Psi)(X^{\text{hor}})$. Now, using that $\Phi_{2}(\xi_{1}) = -\Phi_{1}(\xi_{2})$, the last equality on the right-hand side of Eq. \ref{LocEq2} becomes
\begin{equation}
\label{fundamentalidentitie0}
(\Phi_{2} \circ \Phi_{2})(X^{\text{hor}}) = (\Psi \circ \Psi)(X^{\text{hor}}) + \eta_{2}(\Phi_{1}(X^{\text{hor}}))\Phi_{1}(\xi_{2}),
\end{equation}
Now, we claim that:
\begin{equation}
\label{fundamentalidentitie}
(\Psi \circ \Psi)(Y) = - Y + \eta_{2}(Y)\xi_{2} - (\eta_{2}(\Phi_{1}(Y)) \Phi_{1}(\xi_{2}),
\end{equation}
for any vector field $Y$ on $Q(L)$. In fact, by considering the local expression of $\Psi$, $\forall Y \in \mathfrak{X}(\pi_{Q}^{-1}(U_{i}))$, we have
\begin{center}
$(\Psi \circ \Psi)(Y) = \sin(\phi_{i})\cos(\phi_{i}) \Big ( (H_{i} \circ G_{i} + G_{i} \circ H_{i}) \pi_{Q \ast}(Y)\Big)^{\#} + \cos^{2}(\phi_{i})(G_{i}^{2}\pi_{Q \ast}(Y))^{\#} + \sin^{2}(\phi_{i})(H_{i}^{2}\pi_{Q \ast}(Y))^{\#}.$
\end{center}
Since $H_{i} \circ G_{i} + G_{i} \circ H_{i} = 0,$ and $H_{i} \circ H_{i} = G_{i} \circ G_{i}= - {\rm{Id}} + u_{i} \otimes A_{i} + v_{i} \otimes B_{i},$ we obtain 
\begin{center}
$(\Psi \circ \Psi)(Y)  = -Y + \pi_{Q}^{\ast}(u_{i})(Y)A_{i}^{\#} + \pi_{Q}^{\ast}(v_{i})(Y)B_{i}^{\#}.$
\end{center}
Now, using the local description of $\eta_{2}$, $\xi_{2}$, $\eta_{2} \circ \Phi_{1}$, and $\Phi_{1}(\xi_{2})$ (see Eq. \ref{LocEq1}), it follows that 
\begin{center}
$\eta_{2}(Y)\xi_{2} - (\eta_{2}(\Phi_{1}(Y)) \Phi_{1}(\xi_{2}) = \pi_{Q}^{\ast}(u_{i})(Y)A_{i}^{\#} + \pi_{Q}^{\ast}(v_{i})(Y)B_{i}^{\#}.$
\end{center}
Thus, Eq. \ref{fundamentalidentitie} holds for any vector field $Y$ on $Q(L)$. Hence, taking $Y = X^{\text{hor}}$ in Eq. \ref{fundamentalidentitie}, and replacing the result in Eq. \ref{fundamentalidentitie0}, we obtain
\begin{equation}
\label{horizontalexpression}
(\Phi_{2} \circ \Phi_{2})(X^{\text{hor}}) = - X^{\text{hor}} + \eta_{2}(X^{\text{hor}})\xi_{2}.
\end{equation}
From Eq. \ref{verticalexpression}, and Eq. \ref{horizontalexpression}, we conclude that $\Phi_{2} \circ \Phi_{2} = - {\rm{Id}} + \eta_{2} \otimes \xi_{2}$. Thus, $(\Phi_{2},\xi_{2},\eta_{2})$ defines an almost contact structure on $Q(L)$.
\end{proof}

Now, we can prove our main theorem:

\begin{theorem}
\label{Mainteorem}
Let $(Z, \mathscr{J},\theta)$ be a complex contact manifold of complex dimension $2n+1 \geq 3$. Then there exists a $\rm{U}(1)$-principal bundle $Q$ over $Z$ which admits an almost contact metric 3-structure $(g_{Q},\Phi_{\alpha},\xi_{\alpha},\eta_{\alpha})$, $\alpha = 1,2,3$, satisfying the following properties:
\begin{enumerate}
\item $(\Phi_{1},\xi_{1},\eta_{1})$ is a normal almost contact structure, such that $Z = Q/\mathscr{F}_{\xi_{1}}$, and $\mathscr{L}_{\xi_{1}}g_{Q} = 0$;
\item $\eta_{2}$ and $\eta_{3}$ are contact structures, such that $\eta_{2} \wedge (d\eta_{2})^{2n+1} = \eta_{3} \wedge (d\eta_{3})^{2n+1} \neq 0$;
\item $(g_{Q},\Phi_{\alpha},\xi_{\alpha},\eta_{\alpha})$ is a contact metric structure, for $\alpha = 2,3$.
\end{enumerate}
Moreover, both $Q$ and $(g_{Q},\Phi_{\alpha},\xi_{\alpha},\eta_{\alpha})$, $\alpha = 1,2,3$, can be constructed in a natural way from $Z$ and $\theta$. 
\end{theorem}

\begin{proof}
Consider $Q := Q(L)$, and $(\Phi_{\alpha},\xi_{\alpha},\eta_{\alpha})$, $\alpha = 1,2,$ such that $(\Phi_{1},\xi_{1},\eta_{1})$ is the almost contact structure induced by the IK-connection described in Remark \ref{IKconnection}, and $(\Phi_{2},\xi_{2},\eta_{2})$ is the almost contact structure described in Lemma \ref{Kobayashialmostcontact}. In what follows, we first prove that $Q$ admits an almost contact $3$-structure satisfying (1) and (2), and after that, we prove the existence of a compatible Riemannian metric $g_{Q}$ satisfying the desired properties.

In order to prove that $Q(L)$ admits an almost contact $3$-structure, from Theorem \ref{Kuotheorem}, it is enough to show that $(\Phi_{\alpha},\xi_{\alpha},\eta_{\alpha})$, $\alpha = 1,2,$ satisfy Eq. \ref{L1eq}, and Eq. \ref{L2eq}. As we have mentioned in Remark \ref{firstproperties}, the equations $\Phi_{1}(\xi_{2}) = -\Phi_{2}(\xi_{1})$ and $\eta_{1}(\xi_{2}) = \eta_{2}(\xi_{1}) = 0$, can be easily obtained from the definition of $(\Phi_{\alpha},\xi_{\alpha},\eta_{\alpha})$, $\alpha = 1,2$. Thus, in Eq. \ref{L1eq}, it remains to show that $\eta_{1} \circ \Phi_{2} = - \eta_{2} \circ \Phi_{1}$. This last equation can be easily verified as follows: Given $X \in \mathfrak{X}(Q(L))$, from the definition of $\Phi_{2}$ and $\eta_{1}$, we obtain
\begin{center}
$\eta_{1}(\Phi_{2}(X)) = \eta_{1}(\Psi(X)) - \eta_{1}(X)\eta_{1}(\Phi_{1}(\xi_{2})) - \eta_{2}(\Phi_{1}(X))\eta_{1}(\xi_{1}) = - \eta_{2}(\Phi_{1}(X))$.
\end{center}
Therefore, we have that $(\Phi_{\alpha},\xi_{\alpha},\eta_{\alpha})$, $\alpha = 1,2,$ satisfy Eq. \ref{L1eq}. In order to verify that Eq. \ref{L2eq} holds, we proceed as follows: Given $X \in \mathfrak{X}(Q)$, we notice that
\begin{equation}
\label{AB}
\Phi_{1}(\Phi_{2}(X)) = \Phi_{1}(\Psi(X)) - \eta_{1}(X) \Phi_{1}(\Phi_{1}(\xi_{2})) = \Phi_{1}(\Psi(X)) + \eta_{1}(X) \xi_{2},
\end{equation}
here we have used that $\eta_{1}(\xi_{2}) = 0$. Also, we have
\begin{equation}
\label{BA}
\Phi_{2}(\Phi_{1}(X)) = \Psi(\Phi_{1}(X)) - \eta_{2}(\Phi_{1}(\Phi_{1}(X))) \xi_{1} = \Psi(\Phi_{1}(X)) + \eta_{2}(X) \xi_{1},
\end{equation}
here we have used that $\eta_{2}(\xi_{1}) = 0$. Now, we consider the following fact:
\begin{claim}
\label{anticommutation}
$\Phi_{1}(\Psi(X)) = - \Psi(\Phi_{1}(X)).$
\end{claim}
\begin{proof}
From the definition of $\Psi$, it is enough to prove that the above equation holds locally. Thus, given $Y \in \mathfrak{X}(\pi_{Q}^{-1}(U_{i}))$, for some $U_{i} \in \mathscr{U}$, we have that
\begin{center}
$\Phi_{1}(\Psi(Y)) = \cos(\phi_{i})\big ((\mathscr{J} \circ G_{i})\pi_{Q \ast}(X) \big )^{\#} + \sin(\phi_{i}) \big ((\mathscr{J} \circ H_{i})\pi_{Q \ast}(X) \big )^{\#}.$
\end{center}
Since $G_{i} \circ \mathscr{J} = - \mathscr{J} \circ G_{i}$, and $H_{i} = G_{i} \circ \mathscr{J}$, we obtain
\begin{equation}
\label{AB2}
\Phi_{1}(\Psi(Y)) = \sin(\phi_{i})(G_{i}\pi_{Q \ast}(Y))^{\#} - \cos(\phi_{i})(H_{i}\pi_{Q \ast}(Y))^{\#}.
\end{equation}
On the other hand, we have 
\begin{center}
$\Psi(\Phi_{1}(Y)) = \cos(\phi_{i})\big ((G_{i} \circ \mathscr{J})\pi_{Q \ast}(Y) \big )^{\#} + \sin(\phi_{i}) \big ((H_{i} \circ \mathscr{J})\pi_{Q \ast}(Y) \big )^{\#},$
\end{center}
so it follows that 
\begin{equation}
\label{BA2}
\Psi(\Phi_{1}(Y)) = \cos(\phi_{i})(H_{i}\pi_{Q \ast}(Y))^{\#} - \sin(\phi_{i})(G_{i}\pi_{Q \ast}(Y))^{\#}.
\end{equation}
Therefore, from Eq. \ref{AB2}, adn Eq. \ref{BA2}, we obtain $\Phi_{1}(\Psi(Y)) = - \Psi(\Phi_{1}(Y))$, for any $Y \in \mathfrak{X}(\pi_{Q}^{-1}(U_{i}))$, and any $U_{i} \in \mathscr{U}$, so the desired equation holds for any $X \in \mathfrak{X}(Q(L))$. It concludes the proof of Claim \ref{anticommutation}.\end{proof}

Now, from Eq. \ref{AB}, Eq. \ref{BA}, and Claim \ref{anticommutation}, it follows that
\begin{center}
$\Phi_{1} \circ \Phi_{2} - \eta_{1} \otimes \xi_{2} = \Phi_{1} \circ \Psi = - \Psi \circ \Phi_{1} = - \Phi_{2} \circ \Phi_{1} + \eta_{2} \otimes \xi_{1}$.
\end{center}
Thus, we have that $(\Phi_{\alpha},\xi_{\alpha},\eta_{\alpha})$, $\alpha = 1,2,$ satisfy Eq. \ref{L1eq}. From Theorem \ref{Kuotheorem}, we conclude that $Q(L)$ admits an almost contact $3$-structure $(\Phi_{\alpha},\xi_{\alpha},\eta_{\alpha})$, $\alpha = 1,2,3$. 

From Theorem \ref{almostcircle}, we have that $(\Phi_{1},\xi_{1},\eta_{1})$ is a normal almost contact structure, so we obtain item $(1)$. Further, by construction, we have that $(Q(L),\eta_{2})$ is a contact manifold \cite{Kobayashi}, see also Section \ref{Kobayashi'scontact}. Thus, by considering $(\Phi_{3},\xi_{3},\eta_{3})$ obtained from Eq. \ref{thirdalmostcontact}, in order to prove item (2), it remains to show that $(Q(L),\eta_{3})$ is a contact manifold, and $\eta_{2} \wedge (d\eta_{2})^{2n+1} = \eta_{3} \wedge (d\eta_{3})^{2n+1} \neq 0$. In order to see that, notice that, from Eq. \ref{thirdalmostcontact}, and Eq. \ref{LocEq1}, we have
\begin{equation}
\eta_{3} = -\eta_{2} \circ \Phi_{1} = \sin(\phi_{i})\pi_{Q}^{\ast}(u_{i}) - \cos(\phi_{i})\pi_{Q}^{\ast}(v_{i}), \ \ \ {\text{on}} \ \ \ \pi_{Q}^{-1}(U_{i}).
\end{equation}
Thus, from Eq. \ref{realcontact}, we obtain
\begin{equation}
\label{imafinaryholomorphic}
\eta_{3} = \frac{\mathfrak{Im}\big({\rm{e}}^{\sqrt{-1}\phi_{i}} \pi_{Q}^{\ast}(\theta_{i})\big)}{\sqrt{h_{i} \circ \pi_{Q}}} = \frac{1}{2\sqrt{-1}}(\vartheta - \overline{\vartheta})|_{Q(L)},
\end{equation}
where $\vartheta =z_{i} \pi^{\ast}\theta_{i}$, on $\pi^{-1}(U_{i})$, see Eq. \ref{localholomorphcform}. As in the case of $\eta_{2}$, one can show that $\eta_{3} \wedge (d\eta_{3})^{2n+1} \neq 0$, on $Q(L)$. In fact, a direct computation shows that 
\begin{equation}
(\vartheta - \overline{\vartheta}) \wedge (d\vartheta - d\overline{\vartheta})^{2n+1} = (-1)^{n+1}C_{0}(z_{i}\overline{z_{i}})^{n}(\overline{z}_{i}dz_{i} - z_{i}d\overline{z}_{i}) \wedge \pi^{\ast}\big ( \theta_{i} \wedge \big (d\theta_{i} \big)^{n} \wedge \overline{\theta}_{i} \wedge \big (d\overline{\theta}_{i} \big)^{n} \big ),
\end{equation}
such that $C_{0} = \frac{(2n+1)!}{n!n!}$. Now, by using Eq. \ref{hermitianlocalrelation}, we have
\begin{equation}
z_{i} \overline{z}_{i} = \frac{1}{h_{i}} \Longrightarrow \overline{z}_{i}dz_{i} - z_{i}d\overline{z}_{i} = 2\overline{z}_{i}dz_{i} + \frac{dh_{i}}{h_{i}^{2}}.
\end{equation}
From the above relation, and a similar argument as in \cite{Kobayashi}, we obtain that
\begin{equation}
(\vartheta - \overline{\vartheta}) \wedge (d\vartheta - d\overline{\vartheta})^{2n+1}|_{Q(L)} = (-1)^{n+1}\frac{2C_{0}}{h_{i}^{n}}\overline{z}_{i}dz_{i}  \wedge \pi_{Q}^{\ast}\big ( \theta_{i} \wedge \big (d\theta_{i} \big)^{n} \wedge \overline{\theta}_{i} \wedge \big (d\overline{\theta}_{i} \big)^{n} \big ),
\end{equation}
is different from zero in every point of $Q(L)$. Thus, the pair $(Q(L),\eta_{\alpha})$ is a contact manifold for $\alpha = 2,3$. Moreover, since
\begin{equation}
(\vartheta + \overline{\vartheta}) \wedge (d\vartheta + d\overline{\vartheta})^{2n+1} = C_{0}(z_{i}\overline{z_{i}})^{n}(\overline{z}_{i}dz_{i} - z_{i}d\overline{z}_{i}) \wedge \pi^{\ast}\big ( \theta_{i} \wedge \big (d\theta_{i} \big)^{n} \wedge \overline{\theta}_{i} \wedge \big (d\overline{\theta}_{i} \big)^{n} \big ),
\end{equation}
cf. \cite{Kobayashi}, it follows that $\eta_{2} \wedge (d\eta_{2})^{2n+1} = \eta_{3} \wedge (d\eta_{3})^{2n+1}$, which concludes the proof of item (2).

Let $(Q,\eta_{\alpha})$, $\alpha = 2,3,$ be the contact manifolds describe above. As we have seen, from Theorem \ref{IshiharaKonishi}, we have a Hermitian metric $g_{Z}$ on $Z$ which is associated to the complex almost contact structure induced by $\theta  \in H^{0}(Z,  \Omega_{Z}^{1}\otimes E)$. From this associated Hermitian metric $g_{Z}$, and considering the almost complex structure $(\Phi_{1},\xi_{1},\eta_{1})$ on $Q$ induced from the IK-connection, we define
\begin{equation}
\label{submersionmetric}
g_{Q} := \pi_{Q}^{\ast}(g_{Z}) + \eta_{1} \otimes \eta_{1}.
\end{equation}
In order to prove that $g_{Q}$ is compatible with $(\Phi_{\alpha},\xi_{\alpha},\eta_{\alpha})$, $\alpha = 1,2,3$, we firstly will show that item (3) holds. 
\begin{claim}
\label{claim2}
$\eta_{\alpha}(X) = g_{Q}(X,\xi_{\alpha})$, and $d\eta_{\alpha} = g_{Q}(\Phi_{\alpha}(X),Y)$,  $\forall X,Y \in \mathfrak{X}(Q)$, $\alpha = 2,3$.
\end{claim}
\begin{proof} At first, notice that $g_{Q}(X,\xi_{\alpha}) = g_{Z}(\pi_{Q \ast}X,\pi_{Q \ast} \xi_{\alpha})$, for $\alpha = 2,3$. Now, since (locally)
\begin{center}
$\xi_{2} =  \cos(\phi_{i})A_{i}^{\#} + \sin(\phi_{i})B_{i}^{\#} \ \ \ {\text{and}} \ \ \ \xi_{3} = \sin(\phi_{i})A_{i}^{\#} - \cos(\phi_{i})B_{i}^{\#} $,
\end{center}
using the fact that $u_{i}(X) = g_{Z}(A_{i},X)$, $v_{i} = u_{i} \circ \mathscr{J}$, and that (locally)
\begin{equation}
\label{LocEq3}
\eta_{2} =  \cos(\phi_{i})\pi_{Q}^{\ast}(u_{i}) + \sin(\phi_{i})\pi_{Q}^{\ast}(v_{i}) \ \ \ {\text{and}} \ \ \ \eta_{3} = \sin(\phi_{i})\pi_{Q}^{\ast}(u_{i}) - \cos(\phi_{i})\pi_{Q}^{\ast}(v_{i}),
\end{equation}
one can easily verify that $g_{Q}(X,\xi_{\alpha}) = \eta_{\alpha}(X)$, for $\alpha =2,3$. Now, in order to verify that  $d\eta_{\alpha} = g_{Q}(\Phi_{\alpha} \otimes {\rm{Id}} )$, for $\alpha = 2,3$, we observe that, from Eq. \ref{LocEq3}, it follows that
\begin{equation}
\begin{cases} d\eta_{2} = - d\phi_{i} \wedge \eta_{3} + \cos(\phi_{i}) \pi_{Q}^{\ast}(du_{i}) + \sin(\phi_{i}) \pi_{Q}^{\ast}(dv_{i}),\\
  d\eta_{3} = d\phi_{i} \wedge \eta_{2} + \sin(\phi_{i})\pi_{Q}^{\ast}(du_{i})  - \cos(\phi_{i})\pi_{Q}^{\ast}(dv_{i}).\end{cases}
\end{equation}
From these, since $du_{i} = g_{Z}(G_{i} \otimes {\rm{Id}}) + \sigma_{i} \wedge v_{i}$, and $dv_{i} = g_{Z}(H_{i} \otimes {\rm{Id}}) - \sigma_{i} \wedge u_{i}$. see Eq. \ref{locendomorphisms} and Eq. \ref{compatiblemetric}, it follow that 
\begin{equation}
\label{compatibletwoform}
\begin{cases}d\eta_{2} = -\eta_{1} \wedge \eta_{3} + \cos(\phi_{i})\pi_{Q}^{\ast}\big (g_{Z}(G_{i} \otimes {\rm{Id}}) \big ) + \sin(\phi_{i}) \pi_{Q}^{\ast}\big (g_{Z}(H_{i} \otimes {\rm{Id}}) \big ),\\
d\eta_{3} = \eta_{1} \wedge \eta_{2} + \sin(\phi_{i})\pi_{Q}^{\ast}\big (g_{Z}(G_{i} \otimes {\rm{Id}}) \big ) - \cos(\phi_{i}) \pi_{Q}^{\ast}\big (g_{Z}(H_{i} \otimes {\rm{Id}}) \big ).\end{cases}
\end{equation}
Notice that, locally, $\eta_{1} = \pi_{Q}^{\ast}(\sigma_{i}) + d\phi_{i}$, see Eq. \ref{localIK}. On the other hand, since $\eta_{3} = \eta_{1} \circ \Phi_{2}$, and $\eta_{2} = - \eta_{1} \circ \Phi_{3}$, we have
\begin{equation}
\label{compatibleriemannian}
\begin{cases}g_{Q}(\Phi_{2} \otimes {\rm{Id}}) = (\pi_{Q}^{\ast}g_{Z})(\Phi_{2} \otimes {\rm{Id}}) + \eta_{3} \otimes \eta_{1},\\
g_{Q}(\Phi_{3} \otimes {\rm{Id}}) = (\pi_{Q}^{\ast}g_{Z})(\Phi_{3} \otimes {\rm{Id}}) - \eta_{2} \otimes \eta_{1}.\end{cases}
\end{equation}
Now, from Eq. \ref{thirdalmostcontact}, and Eq. \ref{secondendomorphism}, we can write
\begin{center}
$\Phi_{2} = \Psi - \eta_{1} \otimes \xi_{3} + \eta_{3} \otimes \xi_{1}$ \ \ \ and  \ \ \ $\Phi_{3} = \Phi_{1} \circ \Psi + \eta_{1} \otimes \xi_{2} - \eta_{2} \otimes \xi_{1}$.
\end{center}
Using the above identities, from the local description of $\Psi $ (Lemma \ref{localtensors}), and $\Phi_{1} \circ \Psi$ (Eq. \ref{AB2}), we obtain
\begin{enumerate}
\item $(\pi_{Q}^{\ast}g_{Z})(\Phi_{2} \otimes {\rm{Id}}) = \cos(\phi_{i})\pi_{Q}^{\ast}\big (g_{Z}(G_{i} \otimes {\rm{Id}}) \big ) + \sin(\phi_{i}) \pi_{Q}^{\ast}\big (g_{Z}(H_{i} \otimes {\rm{Id}}) \big ) - \eta_{1} \otimes \eta_{3}$,
\item $(\pi_{Q}^{\ast}g_{Z})(\Phi_{3} \otimes {\rm{Id}}) = \sin(\phi_{i})\pi_{Q}^{\ast}\big (g_{Z}(G_{i} \otimes {\rm{Id}}) \big ) - \cos(\phi_{i}) \pi_{Q}^{\ast}\big (g_{Z}(H_{i} \otimes {\rm{Id}}) \big ) + \eta_{1} \otimes \eta_{2}$. 
\end{enumerate}
Thus, replacing the above expressions in Eq. \ref{compatibleriemannian}, and comparing with Eq. \ref{compatibletwoform}, we have that 
\begin{center}
$d\eta_{2} = g_{Q}(\Phi_{2} \otimes {\rm{Id}} )$ \ \ \ and \ \ \ $d\eta_{3} = g_{Q}(\Phi_{3} \otimes {\rm{Id}} )$.
\end{center}
Notice that, from the last equations, we also obtain that $\xi_{\alpha}\lrcorner d\eta_{\alpha} = 0$, $\alpha = 2,3$, so $\xi_{\alpha}$ is in fact the characteristic vector field associated to the contact structure $\eta_{\alpha}$, $\alpha = 2,3$. It concludes the proof of Claim \ref{claim2}. \end{proof} 

From above, we have that $(g_{Q},\Phi_{\alpha},\xi_{\alpha},\eta_{\alpha})$ defines a contact metric structure on $Q$, and particularly an almost contact metric structure, for $\alpha = 2,3$. Now, to conclude the proof, it remains to verify the compatibility of $g_{Q}$ with the almost contact structure $(\Phi_{1},\xi_{1},\eta_{1})$, i.e., we need to show that 
\begin{center}
$g_{Q}(\xi_{1},X) = \eta_{1}(X), \ \ \ g_{Q}(\Phi_{1}(X),\Phi_{1}(Y)) = g_{Q}(X,Y) - \eta_{1}(X)\eta_{1}(Y),$ \ \ \ \ ($\forall X,Y \in \mathfrak{X}(Q)$).
\end{center}
Since $\eta_{1}(\xi_{1}) = 1$, and $\Phi_{1}( \xi_{1}) = 0$, it is enough to show that the second equation on the right-hand side above holds. From the definition of $g_{Q}$, and the definition of $(\Phi_{1},\xi_{1},\eta_{1})$ (see Remark \ref{IKconnection}), we have
\begin{center}
$g_{Q}(\Phi_{1}(X),\Phi_{1}(Y)) = g_{Z}(\mathscr{J} \pi_{Q \ast}X,\mathscr{J} \pi_{Q \ast}X) = g_{Z}(\pi_{Q \ast}X,\pi_{Q \ast}X) = g_{Q}(X,Y) - \eta_{1}(X)\eta_{1}(Y),$
\end{center}
$\forall X,Y \in \mathfrak{X}(Q)$. Hence, we have that $(g_{Q},\Phi_{1},\xi_{1},\eta_{1})$ defines an almost contact metric structure on $Q$. Further, since $d\eta_{1} = \pi_{Q}^{\ast}(\omega)$, such that $\omega \in \Omega^{1,1}(Z)$ (see Eq. \ref{curvature}), $\eta_{1}(\xi_{1}) = 1$, and $\pi_{Q \ast} \xi_{1} = 0$, we have that $\mathscr{L}_{\xi_{1}}g_{Q} = 0$. Hence, we obtain the desired almost contact metric 3-structure $(g_{Q},\Phi_{\alpha},\xi_{\alpha},\eta_{\alpha})$, $\alpha = 1,2,3$.
\end{proof}

\subsection{Proof of Corollary \ref{C1}}
The following result can be obtained directly from the previous Theorem:
\begin{corollary}
Under the hypotheses of Theorem \ref{Mainteorem}, for every $s = (a,b,c) \in S^{2}$, we have an almost contact metric structure $(g_{Q},\Phi_{s},\xi_{s},\eta_{s})$ on $Q$, such that 
\begin{equation}
\Phi_{s} = a\Phi_{1} + b\Phi_{2} + c\Phi_{3}, \ \ \ \xi_{s} = a\xi_{1} + b\xi_{2} + c\xi_{3}, \ \ \ \eta_{s} = a\eta_{1} + b\eta_{2} + c\eta_{3}.
\end{equation}
Moreover, by considering $\nu_{s} : = g_{Q}(\Phi_{s}\otimes{\rm{Id}})$, we have $\eta_{s} \wedge (\nu_{s})^{2n+1} = \eta_{s'} \wedge (\nu_{s'})^{2n+1} \neq 0$, for all $s,s' \in S^{2}$.
\end{corollary}
\begin{proof}
This result follows from \cite[Theorem 4.3]{Cappeletti}.
\end{proof}

\subsection{Proof of Corollary \ref{Corollary2}} Another consequence of Theorem \ref{Mainteorem} is the following result:

\begin{corollary}
In the setting of Theorem \ref{Mainteorem}, $(Z, \mathscr{J},\theta)$ admits a K\"{a}hler-Einstein metric with positive scalar curvature if at least one of the following (equivalent) conditions holds:
\begin{enumerate}
\item $\Phi_{1} = \nabla \xi_{1}$, where $\nabla$ is the Levi-Civita connection of $g_{Q}$;
\item $\big [ \Phi_{\alpha} ,\Phi_{\alpha} \big ] + 2d\eta_{\alpha} \otimes \xi_{\alpha} = 0$, for $\alpha = 2$ or $\alpha = 3$.
\end{enumerate}
In particular, if $Z$ is compact, and $(1)$ or $(2)$ holds, then $(Z, \mathscr{J},\theta)$ is the twistor spaces of a positive quaternionic K\"{a}hler manifold.
\end{corollary}
\begin{proof}
By considering $(\Phi_{1},\xi_{1},\eta_{1})$ as in Theorem \ref{Mainteorem}, since $\eta_{1}(X) = g_{Q}(X,\xi_{1})$, $\forall X \in \mathfrak{X}(Q)$, and $\mathscr{L}_{\xi_{1}}g_{Q} = 0$, it follows that 
\begin{center}
$\displaystyle 2d\eta_{1}(X,Y) =  X(\eta_{1}(Y)) - Y(\eta_{1}(X)) - \eta_{1}([X,Y]) = g_{Q}(\nabla_{X}\xi_{1},Y) - g_{Q}(\nabla_{Y}\xi_{1},X) = 2g_{Q}(\nabla_{X}\xi_{1},Y)$.
\end{center}
Hence, if $\Phi_{1} = \nabla \xi_{1}$, it follows that $d\eta_{1} = g_{Q}(\Phi_{1} \otimes {\rm{Id}} )$, which in turn implies that $(g_{Q},\Phi_{\alpha},\xi_{\alpha},\eta_{\alpha})$, $\alpha = 1,2,3$, is a contact metric 3-structure. From Theorem \ref{Kashiwada}, it follows that $(g_{Q},\Phi_{\alpha},\xi_{\alpha},\eta_{\alpha})$, $\alpha = 1,2,3$, is in fact a 3-Sasaki structure on $Q$. Notice that, in particular, it follows that $(1) \Rightarrow (2)$. Applying Theorem \ref{Kashiwada2}, we have that $(g_{Q},\Phi_{1},\xi_{1},\eta_{1})$ is, particularly, a regular Sasaki-Einstein structure, so the transverse K\"{a}hler-Einstein structure $(\mathscr{D}_{1} := \ker(\eta_{1}), \Phi_{1}|_{\mathscr{D}_{1}},g_{Q}^{T})$ pushes down to the K\"{a}hler-Einstein structure $(\mathscr{J},g_{Z})$ with positive scalar curvature on $Z = Q/\mathscr{F}_{\xi_{1}}$, see Remark \ref{SEandKE}, Remark \ref{IKconnection}, and Eq. \ref{submersionmetric}. Thus, if $(1)$ holds, we have that $(Z, \mathscr{J},\theta)$ admits a K\"{a}hler-Einstein metric with positive scalar curvature. Now, if one supposes that $(2)$ holds, since $(\Phi_{1},\xi_{1},\eta_{1})$ is normal, it follows from Theorem \ref{3normality} that $(g_{Q},\Phi_{\alpha},\xi_{\alpha},\eta_{\alpha})$, $\alpha = 2,3$, are both Sasakian structures. Hence, from Theorem \ref{2sasaki}, we have that $(g_{Q},\Phi_{\alpha},\xi_{\alpha},\eta_{\alpha})$, $\alpha = 1,2,3$, is in fact a 3-Sasakian structure. Notice that, in particular, it follows that $(2) \Rightarrow (1)$. As before, we have that $(g_{Q},\Phi_{1},\xi_{1},\eta_{1})$ is a regular Sasaki-Einstein structure, so a similar argument shows that $(\mathscr{D}_{1} := \ker(\eta_{1}), \Phi_{1}|_{\mathscr{D}_{1}},g_{Q}^{T})$ pushes down to the K\"{a}hler-Einstein structure with positive scalar curvature on $Z$. Now, if $Z$ is compact, and satisfies (1) or (2), it follows that $Z$ is a Fano contact K\"{a}hler-Einstein manifold. Thus, from Theorem \ref{LeBrun}, we have that $Z$ is the twistor space of a positive quaterionic K\"{a}hler manifold.

\end{proof}

\subsection{Proof of Corollary \ref{3contactfano}} In the setting of Fano contact manifolds, we have the following consequence:

\begin{corollary}
Let $(Z, \mathscr{J},\theta)$ be a Fano contact manifold of complex dimension $2n+1 \geq 3$. Then there exists a $\rm{U}(1)$-principal bundle $Q$ over $Z$ which admits an almost contact metric 3-structure $(g_{Q},\Phi_{\alpha},\xi_{\alpha},\eta_{\alpha})$, $\alpha = 1,2,3$, satisfying the following properties:
\begin{enumerate}
\item $(\Phi_{1},\xi_{1},\eta_{1})$ is a normal almost contact structure, such that $Z = Q/\mathscr{F}_{\xi_{1}}$, and $\mathscr{L}_{\xi_{1}}g_{Q} = 0$;
\item $(\eta_{1},\eta_{2},\eta_{3})$ is a triple of contact structures, such that $\eta_{2} \wedge (d\eta_{2})^{2n+1} = \eta_{3} \wedge (d\eta_{3})^{2n+1} \neq 0$;
\item $(g_{Q},\Phi_{\alpha},\xi_{\alpha},\eta_{\alpha})$ is a contact metric structure, for $\alpha = 2,3$.
\end{enumerate}
Moreover, both $Q$ and $(g_{Q},\Phi_{\alpha},\xi_{\alpha},\eta_{\alpha})$, $\alpha = 1,2,3$, can be constructed in a natural way from $Z$ and $\theta$. 
\end{corollary}

\begin{proof}
Supposing that $c_{1}(Z) > 0$, since $E^{\otimes (n+1)} \cong \det(TZ)$, it follows that $c_{1}(E) > 0$, so we have $\frac{\omega}{2\pi} \in c_{1}(E)$, such that $\omega \in \Omega^{1,1}(Z)$ defines a K\"{a}hler structure on $Z$. Now, since $L = E^{-1}$, we have $-\frac{\omega}{2\pi} \in c_{1}(L)$, and from the compactness of $Z$, we can find a Hermitian structure $\langle \cdot \ , \cdot \rangle_{L} \colon L \times L \to \mathbbm{C}$, such that the curvature $F_{\nabla}$ of its Chern connection $\nabla$ satisfies 
\begin{equation}
\label{curvatureIK}
\displaystyle \frac{\sqrt{-1}}{2\pi}F_{\nabla} = -\frac{\omega}{2\pi},
\end{equation}
see for instance \cite[Theorem 7.10]{Vosin}. By observing that the Hermitian structure $\langle \cdot \ , \cdot \rangle_{L}$ can be described in terms of local smooth positive functions $q_{i} \colon U_{i} \to \mathbbm{R}^{+}$, $U_{i} \in \mathscr{U}$, satisfying $q_{j} = q_{i}|g_{ij}|^{2}$, on $U_{i} \cap U_{j} \neq \emptyset$, such that $L = \{g_{ij}\}$, on every $U_{i} \in \mathscr{U}$, we have $\nabla = d + \partial \log(q_{i})$, and $F_{\nabla} = - \partial \overline{\partial} \log(q_{i})$. Therefore, if we take $\varpi_{i} := \frac{\theta_{i}}{\sqrt{q_{i}}}$, on every $U_{i} \in \mathscr{U}$, and proceed as in Section \ref{Cplxalmostsection}, from the $\Omega$-structure $\{\Omega_{i}\}$, such that 
\begin{center}
$\displaystyle \Omega_{i} = \partial \varpi_{i} - \frac{1}{2}\partial\log(q_{i}) \wedge \varpi_{i}$,
\end{center}
we obtain a complex almost contact metric structure $(g_{Z},u_{i},v_{i},A_{i},B_{i},G_{i},H_{i})$ on $Z$. In this case, the associated local potentials $\sigma_{i} \in \Omega^{1}(U_{i})$, $U_{i} \in \mathscr{U}$, which define the IK-connection on the sphere bundle $Q(L)$ are given by
\begin{center}
$\displaystyle \sqrt{-1}\sigma_{i} = \frac{1}{2}(\partial - \overline{\partial}) \log(q_{i}),$
\end{center}
that is, we have a connection $1$-form $\sqrt{-1}\eta_{1} \in \Omega^{1}(Q(L);\mathfrak{u}(1))$, such that $\eta_{1} = \pi_{Q}^{\ast}(\sigma_{i}) + d\phi_{i}$, on $Q(L)|_{U_{i}}$, for every $U_{i} \in \mathscr{U}$. Now, using the complex almost contact metric structure induced by $\varpi_{i} = \frac{\theta_{i}}{\sqrt{q_{i}}}$, one can apply the construction provided in the proof of Theorem \ref{Mainteorem} in order to obtain an almost contact metric $3$-structure $(g_{Q},\Phi_{\alpha},\xi_{\alpha},\eta_{\alpha})$, $\alpha = 1,2,3$, on $Q(L)$ satisfying (1) and (3). Moreover, $\eta_{2}$ and $\eta_{3}$ are contact structures, such that $\eta_{2} \wedge (d\eta_{2})^{2n+1} = \eta_{3} \wedge (d\eta_{3})^{2n+1} \neq 0$. Hence, in order to conclude the proof, it remains to show that $\eta_{1} \wedge (d\eta_{1})^{2n+1} \neq 0$. From Eq. \ref{curvature}, we have
\begin{center}
$d\eta_{1} = \pi_{Q}^{\ast}(d\sigma_{i}) = \pi_{Q}^{\ast} \big (\sqrt{-1}\partial \overline{\partial} \log(q_{i}) \big ) = \pi_{Q}^{\ast}(\omega)$.
\end{center}
Thus, since $\omega^{2n+1} \neq 0$, it follows that $\eta_{1} \wedge (d\eta_{1})^{2n+1} \neq 0$.
\end{proof}

\subsection{Proof of Corollary \ref{C2}} Finally, from Theorem \ref{Mainteorem}, we are able to prove the following result:

\begin{corollary}
Let $(Z, \mathscr{J},\theta)$ be a complex contact manifold of complex dimension $2n+1 \geq 3$. Then there exists a $\mathbbm{C}^{\times}$-principal bundle $\mathscr{U}(Z)$ over $Z$ such that $\mathscr{U}(Z)$ admits an almost hyperhermitian structure $(g_{\mathscr{U}},\mathbbm{I}_{1},\mathbbm{I}_{2},\mathbbm{I}_{3})$, satisfying:

\begin{enumerate}
\item $(g_{\mathscr{U}},\mathbbm{I}_{1})$ is a Hermitian structure, i.e., $[ \mathbbm{I}_{1},\mathbbm{I}_{1}] = 0$;

\item $\omega_{\alpha} = g_{\mathscr{U}}(\mathbbm{I}_{\alpha} \otimes {\rm{Id}})$, $\alpha = 2,3$, are symplectic structures;

\item $\Upsilon = \omega_{2} + \sqrt{-1}\omega_{3}$ is a holomprphic symplectic structure on $(\mathscr{U}(Z),\mathbbm{I}_{1})$.
\end{enumerate}
Furthermore, both $\mathscr{U}(Z)$ and $(g_{\mathscr{U}},\mathbbm{I}_{1},\mathbbm{I}_{2},\mathbbm{I}_{3})$ can be constructed in a natural way from $(Z, \mathscr{J},\theta)$. 
\end{corollary}

\begin{proof}
Consider $\mathscr{U}(Z) := {\text{Tot}}(L^{\times})$, such that $L^{-1} = E$ is the contact line bundle associated to $(Z, \mathscr{J},\theta)$. Fixed a Hermitian structure $\langle \cdot \ , \cdot \rangle_{L} \colon L \times L \to \mathbbm{C}$, let $Q(L)$ be the underlying sphere bundle of $(L, \langle \cdot \ , \cdot \rangle_{L})$. From this, we consider the identification $\mathscr{U}(Z) \cong Q(L) \times \mathbbm{R}$, such that

\begin{equation}
\label{identification}
u \in \mathscr{U}(Z) \mapsto \Bigg (\frac{u}{||u||} \ , \ t(u) \Bigg) \in Q(L) \times \mathbbm{R},
\end{equation}
where $||u|| = \sqrt{\langle u \ , u \rangle_{L}}$, and $t(u) = \log(||u||)$, $\forall u \in \mathscr{U}(Z)$. By using the almost contact metric 3-structure $(g_{Q}, \Phi_{\alpha},\xi_{\alpha},\eta_{\alpha})$, $\alpha = 1,2,3$, provided by Theorem \ref{Mainteorem} on $Q = Q(L)$, we can equip $\mathscr{C}(Q(L)) = Q(L) \times \mathbbm{R}$ with an almost hypercomplex structure $(\mathbbm{I}_{1},\mathbbm{I}_{2},\mathbbm{I}_{3})$, such that
\begin{center}
$ \displaystyle \mathbbm{I}_{\alpha}(X) : = \Phi_{\alpha}(X) - \eta_{\alpha}(X)\frac{d}{dt}$, \ \ \ \ $ \displaystyle \mathbbm{I}_{\alpha}\Big (\frac{d}{dt} \Big) := \xi_{\alpha}$, \ \ \ \ $(\alpha = 1,2,3)$
\end{center}
for all $X \in \mathfrak{X}(Q(L))$. Since $(\Phi_{1},\xi_{1},\eta_{1})$ is a normal almost contact structure (see Eq. \ref{Normalintegrable}), it follows that $[\mathbbm{I}_{1} , \mathbbm{I}_{1} ] = 0$. Moreover, under the diffeomorphism \ref{identification}, the complex structure $\mathbbm{I}_{1}$ can be identified with the natural complex structure underlying $\mathscr{U}(Z) = {\text{Tot}}(L^{\times})$. Now, from the compatible Riemannian metric $g_{Q}$, we can define a Riemannian metric $g_{\mathscr{C}}$ on $\mathscr{U}(Z)$, such that
\begin{center}
$g_{\mathscr{C}} := g_{Q} + dt\otimes dt$.
\end{center}
It is straightforward to check that $(g_{\mathscr{C}}, \mathbbm{I}_{1},\mathbbm{I}_{2},\mathbbm{I}_{3})$  defines an almost hyperhermitian structure on $\mathscr{U}(Z)$. By setting $g_{\mathscr{U}} := {\rm{e}}^{t}g_{\mathscr{C}}$, we have that $g_{\mathscr{U}}$ is compatible with $\mathbbm{I}_{\alpha}$, $\alpha = 1,2,3$. In particular, since $[\mathbbm{I}_{1}, \mathbbm{I}_{1} ] = 0$, we have that $(g_{\mathscr{U}},\mathbbm{I}_{1})$ defines a Hermitian structure on $\mathscr{U}(Z)$, so we obtain item (1). Now, since $(g_{Q},\Phi_{\alpha},\xi_{\alpha},\eta_{\alpha})$ is a contact metric structure, for $\alpha = 2,3$, we have
\begin{center}
$\omega_{\alpha} := g_{\mathscr{U}}(\mathbbm{I}_{\alpha} \otimes {\rm{Id}}) = {\rm{e}}^{t} \Theta_{\alpha} = d({\rm{e}}^{t} \eta_{\alpha})$, \ \ \ \ \ \ ($\alpha = 2,3$)
\end{center}
where $\Theta_{\alpha} = g_{\mathscr{C}}(\mathbbm{I}_{\alpha} \otimes {\rm{Id}})$, $\alpha = 2,3$, see Remark \ref{Hyperkahler3contact}. Thus, $d\omega_{2} = d\omega_{3} = 0$, so we obtain item (2). 

In order to prove item (3), firstly, we observe that
\begin{center}
$\eta_{2} = \frac{1}{2}(\vartheta + \overline{\vartheta})|_{Q(L)} \ \ \ {\text{and}} \ \ \ \eta_{3} = \frac{1}{2\sqrt{-1}}(\vartheta - \overline{\vartheta})|_{Q(L)},$
\end{center}
where $\vartheta \in \Omega^{1}_{\mathscr{U}(Z)}$ is the holomorphic 1-form locally described by $\vartheta =z_{i} \pi^{\ast}\theta_{i}$, on $\pi^{-1}(U_{i})$, see Eq. \ref{localholomorphcform}. By keeping the same notation as in Section \ref{Kobayashi'scontact}, we have ${\rm{e}}^{t} = (\sqrt{h_{i} \circ \pi})|z_{i}|$, on $\pi^{-1}(U_{i})$. Therefore, by considering polar coordinates $z_{i} = |z_{i}|{\rm{e}}^{\sqrt{-1}\phi_{i}}$, it follows that
\begin{equation}
\displaystyle \vartheta = \frac{{\rm{e}}^{t + \sqrt{-1}\phi_{i}}}{\sqrt{h_{i}\circ \pi}}\pi^{\ast}\theta_{i} = {\rm{e}}^{t} \big (\eta_{2} + \sqrt{-1}\eta_{3}\big),
\end{equation}
see Eq. \ref{localcontactKobayashi}, and Eq. \ref{imafinaryholomorphic}. Hence, since $\omega_{\alpha}= d({\rm{e}}^{t}\eta_{\alpha})$, $\alpha =2,3$, we obtain
\begin{equation}
\Upsilon = \omega_{2} + \sqrt{-1}\omega_{3} = d\vartheta.
\end{equation}
Now, since $\vartheta$ is holomorphic with respect to $\mathbbm{I}_{1}$, it follows that $\Upsilon$ is a holomorphic (2,\ 0)-form on $(\mathscr{U}(Z),\mathbbm{I}_{1})$. Further, since $\theta_{i} \wedge \big (d\theta_{i} \big)^{n} \neq 0$, on $U_{i}$, and 
\begin{center}
$(d\vartheta)^{n+1} = (n+1)z_{i}^{n}dz_{i} \wedge \pi^{\ast} \big (\theta_{i} \wedge (d\theta_{i})^{n}\big)$,
\end{center}
on $\pi^{-1}(U_{i})$, it follows that $\Upsilon^{n+1} \neq 0$ on $\mathscr{U}(Z)$. Therefore, we have that $\Upsilon = \omega_{2} + \sqrt{-1}\omega_{3}$ defines a holomorphic symplectic structure on $(\mathscr{U}(Z),\mathbbm{I}_{1})$.
\end{proof}


\begin{thebibliography}{BGGSM}

\bibitem{Albers} Albers, P.; Geiges, H.; Zehmisch, K.; Reeb dynamics inspired by Katok's example in Finsler geometry, Math. Ann. 370 (2018), 1883-1907.

\bibitem{BEAUVILLE} Beauville, A.; Fano contact manifolds and nilpotent orbits, Comment. Math. Helv. 73, 566-583 (1998).

\bibitem{BEAUVILLE1} Beauville, A.; Riemannian holonomy and algebraic geometry. Enseign. Math. (2) 53 (2007),

\bibitem{Biquard1} Biquard, O.; Gauduchon, P.; Hyper-K\"{a}hler metrics on cotangent bundles of Hermitian symmetric spaces, Geometry
and Physics (Aarhus, 1995), Lecture Notes in Pure and Appl. Math., vol. 184, pp. 287-298, Dekker, New York (1997).

\bibitem{Biquard} Biquard, O.; Sur les \'{e}quations de Nahm et la structure de Poisson des alg\`{e}bres de Lie semi-simples complexes, Math. Ann. 304 (1996) 253-276.

\bibitem{Blair} Blair, David E.; Riemannian Geometry of Contact and Symplectic Manifolds, Progress in Mathematics 203, Birkh\"{a}user Basel (2010).

\bibitem{Boothby2} Boothby, W. M.; Homogeneous complex contact manifolds, Proc. Sympos. Pure Math. Vol. 3, pp. 144-154, Amer. Math. Soc., Providence, R. I., 1961.

\bibitem{Boothby} Boothby, W. M.; Wang, H. C.; On contact manifolds, Ann. of Math., 68 (1958), 721.

\bibitem{Sasakigeometry} Boyer, C.; Galicki, K.; Sasakian Geometry, Oxford Mathematical Monographs, Oxford University Press; 1 edition (2008).

\bibitem{3twistor} Boyer, C.; Galicki, K.; The twistor space of a 3-Sasakian manifold, Internat. J. Math. 8 (1997), 31-60.

\bibitem{3Einstein} Boyer, C.; Galicki, K.; 3-Sasakian manifolds, Surveys in differential geometry: essays on Einstein manifolds, 123-184. Surv. Differ Geom., VI, Int. Press, Boston (1999).

\bibitem{Algtorus} Buczy\'{n}ski, J.; Wi\'{s}niewski, J. A.; weber, A; Algebraic torus actions on contact manifolds. ArXiv:1802.05002, (2018).

\bibitem{CALABIANSATZ} Calabi, E.; Metriques K\"{a}hl\'{e}riennes et fibr\'{e}s holomorphes, Ann. Sci. \'{E}cole Norm. Sup. (4) 12 (1979), 269-294.

\bibitem{Cappeletti} Cappelletti-Montano, B.; De Nicola, A.; Yudin, I.; Cosymplectic p-spheres, J. Geom. Phys. 100 (2016), 68-79.

\bibitem{Demailly} Demailly, J.-P.; On the Frobenius integrability of certain holomorphic p-forms. In Complex geometry (G\"{o}ttingen, 2000), pages 93-98. Springer, Berlin, 2002.


\bibitem{Feix} Feix, B.; Hyperk\"{a}hler metrics on cotangent bundles, J. reine angew. Math. 532 (2001), 33-46.

\bibitem{Foreman} Foreman, B.; Variational Problems on Complex Contact Manifolds with Applications to Twistor Space Theory, thesis, Michigan State University.

\bibitem{GeigesGonzalo} Geiges, H.; Gonzalo, J.; Contact geometry and complex surfaces, Invent. Math., 121 (1995), 147-209.

\bibitem{HATAKEYAMA} Hatakeyama, Y.; Some notes on differentiable manifolds with almost contact structures, Osaka Math. J. (2) 15 (1963), 176-181. MR 27 \#705. 

\bibitem{Hitchin} Hitchin, N. J.; The self-duality equations on a Riemannian surface, Proc.London Math. Soc., 55 (1987), 59-126.

\bibitem{IshiharaKonishi2} Ishihara, S., Konishi, M.: Complex almost contact manifolds. Kodai Math. J. 3, 385-396 (1980). 
 
\bibitem{IshiharaKonishi} Ishihara, S.; Konishi M.; Complex almost contact structures in a complex contact manifold, Kodai Math. J., 5, (1982) 30-37.

\bibitem{IshiharaKonishi1} Ishihara, S.; Konishi M.; Real contact 3-structure and complex contact structure, Southeast Asian Bulletin of Mathematics, 3 (1979), 151-161.

\bibitem{Ishihara} Ishihara, S.; Quaternion K\"{a}hlerian manifolds and fibered Riemannian spaces with Sasakian 3-structure, Kodai Math. Sem. Rep. 25 (1973), 321-329.

\bibitem{Jelonek} Jelonek, W; Positive and negative 3-K-contact structures. Proc. Am. Math. Soc., 129, 247-256 (2001). 

\bibitem{Kaledin} Kaledin, D.; Hyperk\"{a}hler metrics on total spaces of cotangent bundles, alg-geom/9710026, 1997, revised version math.AG/0011256, 2000.

\bibitem{Kashiwada2} Kashiwada, T.; A note on a Riemannian space with Sasakian 3-structure, Nat. Sci. Rep. Ochanomizu Univ. , 22 (1971) pp. 1-2.

\bibitem{Kashiwada} Kashiwada, T.; On a contact 3-structure, Math. Z. 238 (2001), no. 4, 829-832.

\bibitem{Kobayashibundles} Kobayashi, S.; Differential geometry of complex vector bundles. Princeton, New Jersey: Princeton University Press 1987.

\bibitem{Kobayashi} Kobayashi, S.; Remarks on complex contact manifolds. Proc. Am. Math. Soc. 10, 164-167 (1959)

\bibitem{Konishi} Konishi, M.; On manifolds with Sasakian 3-structure over quaternion K\"{a}hlerian manifolds, Kodai Math. Sem. Reps. 26 (1975), 194-200.

\bibitem{Kuo} Kuo, Y. Y.; On almost contact 3-structure, T\^{o}hoku Math. J., 22(1970), 325-332.

\bibitem{Kronheimer2} Kronheimer, P.; A hyperk\"{a}hler structure on the cotangent bundle of a complex Lie group, MSRI preprint 1988, arXiv:math/0409253 (2004).


\bibitem{LEBRUNFANO} LeBrun, C.; Fano manifolds, contact structures and quaternionic geometry, International J. Math. 6 (1995), 419-437.

\bibitem{LeBrunSalamon} LeBrun, C.; Salamon, S.; Strong rigidity of positive quaternion-K\"{a}hler manifolds. Invent. Math., 118(1):109-132, 1994.

\bibitem{Fortourists} LeBrun, C.; Twistors for Tourists: A Pocket Guide for Algebraic Geometers, Proc. Symp. Pure Math. 62.2 (1997) 361-385.

\bibitem{MORIMOTO} Morimoto, A.; On normal almost contact structures, J. Math. Soc. Japan 15 (1963) 420-436.

\bibitem{Nakajima} Nakajima, H.; Instantons on ALE spaces, quiver varieties, and Kac-Moody algebras, Duke Math. J. 76 (1994), 365-416.

\bibitem{Sasaki} Sasaki, S.; On differentiable manifolds with certain structures which are closely related to almost contact structures I, T\^{o}hoku Math, Volume 12, Number 3 (1960), 459-476.

\bibitem{SASAKIHARAKEYMAALMOST} Sasaki, S.; Hatakeyama, Y.; On differentiable manifolds with certain structures which are closely related to almost contact structures II, T\^{o}hoku Math. J., Volume 13, Number 2 (1961), 281-294.

\bibitem{Sakamoto} Sakamoto, K.; On the topology of quaternion K\"{a}hler manifolds, Tohoku Math. J. (2) 26(3), 389-405 (1974).

\bibitem{SALAMONQUATERNIONIC} Salamon, S.; Quarternionic K\"{a}hler Manifolds. Inventiones mathematicae, 67(1):143-171, 1982.

\bibitem{Shibuya} Shibuya, Y.; On the existence of a complex almost contact structure, Kodai Math. J. 1 (1978) 197-204.

\bibitem{Swann} Swann, A.; Hyperk\"{a}ler and quaternionic K\"{a}ler geometry, Math. Ann. 289 (1991), 421-450.

\bibitem{Tanno} Tanno, S.; Remarks on a triple of K-contact structures. T\^{o}hoku Math. J., 48, 519-531 (1996). 

\bibitem{Udriste} Udriste, C.; Structures presque coquaternioniennes, Bull. Math. Soc. Sci. Math. R. S. Roumanie, 13, 487-507.

\bibitem{Vosin} Voisin, C.; Schneps, L.; Hodge Theory and Complex Algebraic Geometry I: Volume 1, Cambridge Studies in Advanced Mathematics, Cambridge University Press; 1 edition (2008).

\bibitem{Wolf} Wolf, J. A.; Complex homogeneous contact manifolds and quaternionic symmetric spaces, J. Math. Mech., 14 (1965), 1033-1047.

\bibitem{Yano} Yano. K.; Ishihara, S.; Konishi, M.; Normality of almost contact 3-structure, T\^{o}hoku Math. J. 25 (1973), 167-175.

\bibitem{Zessin} Zessin, M.; On contact p-spheres. Ann. Inst. Fourier (Grenoble) 55, 1167-1194 (2005)

\end{thebibliography}
\end{document}